\documentclass[numbers,webpdf]{ima-authoring-template}%
\graphicspath{{Fig/}}

\theoremstyle{thmstyletwo}%
\newtheorem{theorem}{Theorem}[section]
\newtheorem{lemma}[theorem]{Lemma}
\newtheorem{corollary}[theorem]{Corollary}
\newtheorem{example}{Example}[section]%
\newtheorem{remark}[theorem]{Remark}

\numberwithin{equation}{section}

\definecolor{pass}{rgb}{0,0,0.7}
\definecolor{og}{rgb}{0.95,0.5,0}
\definecolor{grp}{rgb}{0.58,0.35,0.75}

\newcommand {\dd} {{\mathrm{d}}}
\newcommand{\pt}{\partial}
\newcommand{\eps}{\varepsilon}
\newcommand {\A} {{\mathcal A}}
\newcommand {\B} {{\mathcal B}}
\newcommand {\D} {{\mathcal D}}
\newcommand {\U} {{\mathcal U}}
\newcommand {\beq} {\begin{equation}}
\newcommand {\eeq} {\end{equation}}

\begin{document}

\DOI{DOI HERE}
\copyrightyear{2021}
\vol{00}
\pubyear{2021}
\access{Advance Access Publication Date: Day Month Year}
\appnotes{Paper}
\copyrightstatement{Published by Oxford University Press on behalf of the Institute of Mathematics and its Applications. All rights reserved.}
\firstpage{1}


\title[Backward L2 method for subdiffusion equations]{Error analysis of a backward L2 method on graded meshes for linear and semilinear subdiffusion equations}

\author{Natalia Kopteva*
\address{\orgdiv{Department of Mathematics and Statistics
}, \orgname{University of Limerick}, \orgaddress{
\postcode{V94 T9PX}, 
\country{Ireland}}}}
\author{Lu Wang
\address{\orgdiv{School of Science}, \orgname{Harbin Istitute of Technology, Shenzhen}, \orgaddress{
\postcode{518055}, 
\country{China}}}}

\authormark{Natalia Kopteva and Lu Wang}

\corresp[*]{Corresponding author: \href{email:Natalia.Kopteva@ul.ie}{Natalia.Kopteva@ul.ie}}

\received{Date}{0}{Year}
\revised{Date}{0}{Year}
\accepted{Date}{0}{Year}


\abstract{
Time-fractional linear and semilinear parabolic equations with a Caputo time derivative of order $\alpha\in(0,1)$ are
      considered. Their solutions exhibit an initial singularity of type $t^\sigma$
      for any fixed $\sigma \in (0,1) \cup (1,2)$.
        A backward L2-type discrete fractional-derivative operator of order $3-\alpha$ is
 considered on nonuniform temporal meshes.
 Sufficient conditions for inverse monotonicity of this operator areestablished.
Hence,  sharp pointwise-in-time error bounds are obtained on graded temporal meshes with arbitrary degree of grading.
The theoretical findings are illustrated by numerical experiments.}



\maketitle

\section{Introduction}\label{se:introduction}

In this paper, we consider linear and semilinear time-fractional parabolic equations, also known as subdiffusion equations, involving the Caputo derivative of order $\alpha\in(0,1)$, defined by
    \begin{align}\label{eq:def-Caputoderivative}
        \pt_{t}^{\alpha} u(\cdot,t) := \frac{1}{\Gamma(1-\alpha)} \int_{0}^{t} (t-s)^{-\alpha}\, \partial_{s}u(\cdot,s)  \, \dd s\quad \hbox{for}\ 0<t\leq T.
    \end{align}
Our aim is to  investigate a backward-quadratic-interpolation-based scheme proposed in \citep{GaoSun2014} on graded temporal meshes. More precisely, we consider the discrete fractional derivative operator \(\delta_t^\alpha\) associated with \(\partial_t^\alpha\), defined by applying linear Lagrange interpolation on the first time interval and backward-quadratic Lagrange interpolation on all remaining intervals.
Thus, on the temporal mesh $0=t_{0}<t_{1}<\dots<t_{M}=T$ 
it is defined by
    \begin{align}\label{eq:defdelta}
        \delta_{t}^{\alpha}U^{m}:=\pt_{t}^{\alpha}(\Pi U)(t_{m}),\qquad
        \Pi:=
        \left\{
        	\begin{array}{ll}
        		\Pi_{1,1}\quad & \hbox{on}\  (0,t_{1}),\\
        		\Pi_{2,j}\quad & \hbox{on}\  (t_{j-1},t_{j}),\ j\geq 2.
        	\end{array}
        \right.
    \end{align}
Here, $\Pi_{1,1}$ and $\Pi_{2,j}$ are respectively the standard linear and {\it backward-quadratic} Lagrange interpolation operators with the following interpolation points:
$\Pi_{1,1}:\{ t_{0},t_{1}\}$ and $\Pi_{2,j}:\{ t_{j-2},t_{j-1},t_{j} \}$.

Various interpolation-based schemes have been proposed in the literature for subdiffusion equations, with the L1 method arguably attracting the greatest attention
(see, e.g., \citep{NataliaL1, MartinL1, Liao, NataliaSean}). Notably, most error analyses of L1-type schemes rely on the discrete maximum/comparison principle satisfied by the associated temporal discrete operator, which can be readily verified on any temporal grid, as this operator corresponds to an M-matrix.

Among higher-order schemes, the Alikhanov method (see, e.g., \citep{Alikhanov2015, ChenMartin2019, NataliaMeng,LiaoMcLeanZhang2021}), 
which is of order $3-\alpha$, also enjoys a similar M-matrix property for its temporal discrete operator and has consequently received considerable attention in the literature. The method may be viewed as an analogue of the Crank–Nicolson scheme for time-fractional equations; similarly, the spatial-derivative terms are discretized using two time levels. This limits the order of convergence in time to $2$ and introduces additional challenges in the error analysis for the semilinear case.
Furthermore, the Alikhanov scheme becomes inapplicable if the equation involves multiple time-fractional derivatives of different orders.
It is worth noting that the Alikhanov temporal discretization is based on forward-quadratic interpolation over all time intervals except the final one, where a special linear interpolation is employed.

There are two alternative approaches, leading to two L2-type schemes: forward-quadratic interpolation over all time intervals except the final one, combined with backward-quadratic interpolation on the final interval \citep{NataliaL2, Xu2016,NataliaSemi}; or the above \eqref{eq:defdelta}, backward-quadratic interpolation over all time intervals except the first one \citep{GaoSun2014}.
We shall refer to these schemes as the forward and backward L2 methods, respectively. Importantly, both schemes employ backward-quadratic interpolation on the final time interval, which causes the associated temporal discrete operator to lack the M-matrix property, even on uniform temporal grids (similarly to backward differencing in the context of classical parabolic equations). This feature makes their analysis considerably more challenging and may explain why they have received comparatively less attention in the literature.
The forward L2 method was proposed and analyzed on uniform temporal meshes in \citep{Xu2016}, while its stability and convergence on graded temporal meshes were established in  \citep{NataliaL2} for linear subdiffusion equations and in \citep{NataliaSemi} for the semilinear case, with sharp pointwise-in-time error bounds.

Much less analysis is available for the backward L2 method \eqref{eq:defdelta}, while the  latter is arguably more natural, as it employs the same backward-quadratic interpolation throughout the temporal mesh,  a feature that also makes it intrinsically better suited for the residual-based adaptive time stepping \citep{FranzKoptevaJCOAM}.
%
In \citep{GaoSun2014}, where this method was proposed, the analysis was limited to truncation error estimates on uniform meshes and numerical experiments for subdiffusion equations with artificially smooth solutions.
A somewhat similar backward-quadratic-interpolation-based scheme was introduced and analyzed on uniform temporal meshes in the context of the Riemann–Liouville fractional derivative in \citep{Yan2018}. To achieve the optimal convergence order $3-\alpha$ for nonsmooth solutions in positive time, i.e., away from the initial time, on uniform meshes, the scheme was appropriately modified at the first two time levels.
In \citep{Alikhanov2021}, a version of \eqref{eq:defdelta} with forward-quadratic interpolation on the first time interval was considered on uniform temporal meshes for artificially smooth solutions, namely, solutions that are three times continuously differentiable in time.
The stability and convergence of order $3-\alpha$ were established in the discrete in the discrete $L_2(0,T;\,W_2^1(\Omega))$ norm for the spatial domain $\Omega=(0,1)$.
These results were extended in \citep{Alikhanov2025} to multiterm time-fractional parabolic equations.

Thus, the backward L2 scheme on nonuniform temporal meshes has not yet been analyzed in the literature, while even on uniform meshes, the available error bounds are quite limited.
In this paper, we consider this scheme in the context of linear and semilinear subdiffusion equations, with the linear equation given by\vspace{-0.3cm}
    \begin{align}\label{eq:problem}
    \left.
    \begin{array}{ll}
        &\qquad \pt_{t}^{\alpha} u+\mathcal{L} u = f(x,t)\quad \hbox{for}\ (x,t)\in\Omega\times(0,T],\\
        &\qquad u(x,t) = 0 \quad \hbox{for}\ (x,t)\in\pt\Omega\times(0,T],\qquad u(x,0)=u_{0}(x)\quad \hbox{for}\ x\in\Omega.\vspace{-0.1cm}
    \end{array}
    \right.
    \end{align}
This problem is posed in a bounded Lipschitz domain $\Omega\subset \mathbb{R}^{d}$, where $d\in\{1,2,3\}$.
Here, the linear second-order elliptic spatial operator $\mathcal{L}$ is defined by\vspace{-0.2cm}
    \begin{align}\label{eq:def-ellipticspatial}
        \mathcal{L} u:= \sum_{k=1}^{d} \{ -\partial_{x_{k}}(a_{k}(x)\partial_{x_{k}}u)  \}+c(x)u,\vspace{-0.3cm}
    \end{align}
with sufficiently smooth coefficients $\{a_{k}\}$ and $c$, which are assumed to satisfy $a_{k}>0$ and $c\geq 0$ in~$\bar{\Omega}$ and  $a_{k}, c\in C(\bar{\Omega})$.

The core idea in the analysis of the backward L2 operator \eqref{eq:defdelta} will be a simplified version of the approach developed in \citep{NataliaL2}.
The discrete version
$\delta_{t}^{\alpha}$ of Caputo fractional-derivative operator $\pt_{t}^{\alpha}$ will be represented in the form
    \begin{align}\label{eq:newrepresentation}
        \delta_{t}^{\alpha}  U^{m} = \sum_{j=0}^{m}  \kappa_{m,j} V^{j} \; \forall m\geq 1,\qquad V^{j}:= \frac{U^{j}-\beta U^{j-1}}{1-\beta}\; \forall j\geq 1,\quad V^{0}:=U^{0},
    \end{align}
where a constant $\beta\in[0,1)$.
(Alternatively, one may consider a more general, albeit more intricate, version of the above representation, involving a variable $\beta_j\in[0,1)$ and $V^{j}:= \frac{U^{j}-\beta_{j}U^{j-1}}{1-\beta_{j}}$, as used in  \citep{NataliaL2}.)
Hence,  assuming $U^0=V^0=0$,
we shall formulate relatively simple sufficient conditions for choosing $\beta\in[0,1)$ such that
  \begin{align}\label{eq:condition_kappa}
      \kappa_{m,m}>0 \quad\forall\,m\ge1,\qquad \kappa_{m,j}\leq 0\quad \forall\, 1\leq j\le m-1.
  \end{align}
The representation \eqref{eq:newrepresentation} subject to  \eqref{eq:condition_kappa} immediately implies that $\delta_{t}^{\alpha}$ can be written as the product of two M-matrices and is therefore inverse-monotone,
The resulting stability properties of the discrete fractional-derivative operator  then allow us to employ the general error analysis framework developed in \citep{NataliaMeng, NataliaL2,NataliaUpLow, NataliaSemi, NataliaSean}.

Our results are novel and different from the earlier literature as follows.
\begin{itemize}
  \item
  While the backward L2 scheme on nonuniform temporal meshes has not been addressed in the literature,
  in this paper we establish sharp pointwise-in-time error bounds in the $L_2(\Omega)$  norm
  on graded temporal meshes with arbitrary degree of grading.

  \item
  The error bounds remain sharp both for
  a singular behavior of the exact solution of \eqref{eq:problem} at the initial time of type
  $t^\alpha$ and, more generally, of type
  $t^\sigma$
      for any fixed $\sigma \in (0,1) \cup (1,2)$.
To handle this more general case, we establish our key stability result, Theorem~\ref{th-euniformgrid}, for arbitrarily graded meshes and remove a restriction imposed in \cite{NataliaL2}, as discussed in Remark~\ref{rem_th-euniformgrid}.

\item While the proof of the main stability result, Theorem~\ref{th-euniformgrid},
borrows from \citep{NataliaL2} the idea of using a nonstandard basis
(see Fig.\,\ref{fig_basis}\,(left)), we simplify this approach by using a constant
$\beta$ in \eqref{eq:newrepresentation}. At the same time, we verify a new condition
\eqref{eq:newcondition3}, which has not previously been addressed in the literature
in the context of the L2 operator considered here.

  \item
  Furthermore, we extend our error bounds to the semilinear case, using a simple idea introduced in the proof of \citep[Lemma~2.2]{NataliaSemi}.

\end{itemize}

The paper is organised as follows. Sufficient conditions on the temporal mesh for the inverse-monotonicity of the discrete fractional-derivative operator in the form \eqref{eq:newrepresentation}, \eqref{eq:condition_kappa} are obtained in Section~\ref{se:inv-mono}. The stability properties of \(\delta_t^\alpha\), including the key stability result of Theorem~\ref{th-euniformgrid}, are then established in Section~\ref{sec_stability}. Next, in Section~\ref{sec_trunc}, we estimate the temporal truncation error of the L2 fractional-derivative operator and complete the error analysis for a version of \eqref{eq:problem} without spatial derivatives. Similar error bounds for linear and semilinear subdiffusion equations are given in Sections~\ref{sec_err_linear} and~\ref{sec_semi}, respectively. Our theoretical findings are illustrated by numerical experiments in Section~\ref{sec_numerics}.

{\it Notation.}
We write
 $a\simeq b$ when $a \lesssim b$ and $a \gtrsim b$, and
$a \lesssim b$ when $a \le Cb$ with a generic positive constant $C$ depending on $\Omega$, $T$, $u_0$ and
$f$,
but not 
%
 on the total numbers of degrees of freedom in space or time.
  Also, 
  we shall use the standard norms
  in the space $L_2(\Omega)$ and the related standard Sobolev space $H^1(\Omega)=W_2^1(\Omega)$,
  while  $H^1_0(\Omega)$ is the standard space of functions in $H^1(\Omega)$ vanishing on $\pt\Omega$.


\section{Inverse-monotonicity of the discrete fractional-derivative operator}\label{se:inv-mono}

  \begin{figure}[t!]
\noindent~\hfill%
\begin{tikzpicture}[scale=0.23]
  \draw[->] (-0.2,0) -- (24.5,0) ;
  \draw[->] (0,-0.2) -- (0,11.5) ;
  \node[left] at (0,10) {$1$};
  \node[left] at (0,6) {$\beta\,\,$};
  \node[left] at (0,3.4) {$\beta^2$};
  \node[left] at (0,1.2) {$\beta^3$};
\path[draw, thick] (0,0)--(5.5,0)--(8.5,10)--(12,6)--(16,3)--(21,1.5)--(24,1);

\draw[fill] (0,0) circle [radius=0.3] node[below] {$0$};
\draw[fill] (1.5,0) circle [radius=0.3];
\draw[fill] (3.5,0) circle [radius=0.3];
\draw[fill] (5.5,0) circle [radius=0.3] node[below] {$t_{j-1}$};
\draw[fill] (8.5,10) circle [radius=0.3]node[above right] {$\Phi^j$};
\draw[fill] (16,3) circle [radius=0.3];
\draw[fill] (12,6) circle [radius=0.3];
\draw[fill] (21,1.5) circle [radius=0.3];

\path[draw,help lines]  (-0.3,10)--(8.9,10);
\path[draw,help lines]  (-0.3,6)--(12.4,6);
\path[draw,help lines]  (-0.3,3)--(16.4,3);
\path[draw,help lines]  (-0.3,1.5)--(21.4,1.5);

\path[draw,help lines]  (8.5,-0.3)--(8.5,10.4);
\path[draw,help lines]  (12,-0.3)--(12,6.4);
\path[draw,help lines]  (16,-0.3)--(16,3.4);
\path[draw,help lines]  (21,-0.3)--(21,1.9);

\node[below]  at (8.5,0) {$t_{j}$};
\node[below]  at (12,0) {$t_{j+1}$};
\node[below]  at (16,0) {$t_{j+2}$};
\node[below]  at (21,0) {$t_{j+3}$};
\end{tikzpicture}%
\hfill%
\begin{tikzpicture}[scale=0.23]
  \draw[->] (-0.2,0) -- (18,0) ;
  \draw[->] (0,-0.2) -- (0,11.5) ;
  \node[left] at (0,10) {$1$};
\path[draw, thick] (0,0)--(5.5,0)--(8.5,10)--(12,0)--(16,0)--(17,0);

\draw[fill] (0,0) circle [radius=0.3] node[below] {$0$};
\draw[fill] (1.5,0) circle [radius=0.3];
\draw[fill] (3.5,0) circle [radius=0.3];
\draw[fill] (5.5,0) circle [radius=0.3] node[below] {$t_{j-1}$};
\draw[fill] (8.5,10) circle [radius=0.3]node[above right] {$\phi^j$};
\draw[fill] (16,0) circle [radius=0.3];
\draw[fill] (12,0) circle [radius=0.3];

\path[draw, dotted, very thick] (0,10)--(5.5,10)--(8.5,0)--(12,0)--(16,0)--(17,0);
\draw (5.5,10) circle [radius=0.3]node[above left] {$\psi^j$};
\draw (1.5,10) circle [radius=0.3];
\draw (3.5,10) circle [radius=0.3];
\draw (0,10) circle [radius=0.3];
\draw (8.5,0) circle [radius=0.3];


\node[below]  at (8.5,0) {$t_{j}$};
\node[below]  at (12,0) {$t_{j+1}$};
\node[below]  at (16,0) {$t_{j+2}$};
%
\end{tikzpicture}\hfill~%
 \caption{\label{fig_basis}\it\small
Nonstandard basis $\{\Phi^j\}$  from \eqref{eq:Phi} (left),
 hat-function basis $\{\phi^j\}$ (right, solid),  and
a Heaviside-type  basis $\{\psi^j\}$ (right, dotted).}
 \end{figure}
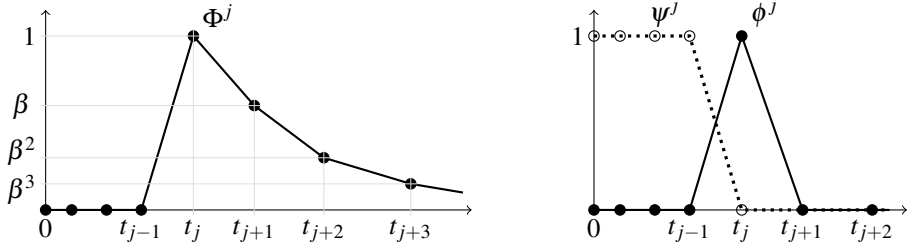
To describe a representation of type \eqref{eq:newrepresentation}  on a non-uniform temporary mesh, it is convenient to employ
a nonstandard basis $\{\Phi^{j}(t_{k})\}_{j=0}^{M}$ (see Fig.\,\ref{fig_basis}\,(left)) for functions in $\mathbb{R}^{M+1}$ associated with the mesh $\{t_{k}\}_{k=0}^{M}$, introduced in \citep{NataliaL2}, which is defined by
    \begin{align}\label{eq:Phi}
        \Phi^{j}(t_{k}) := 0\;\; \hbox{for} \ k\leq  j-1,\quad \Phi^{j}(t_{j}) := 1,\quad \Phi^{j}(t_{k}):= \beta\,\Phi^{j}(t_{k-1})\;\; \hbox{for}\ k\geq j+1.
    \end{align}
We now recall two results from \citep{NataliaL2}, which apply to any discretization that one may want to represent as \eqref{eq:newrepresentation} (which were presented in \citep{NataliaL2} in a more general context of a variable $\beta=\beta_j$).

\begin{lemma}[{\citep[Lemma~2.2]{NataliaL2}}]\label{le:tran}
Given $\beta\in[0,1)$ and the basis \eqref{eq:Phi}, the coefficients $\kappa_{m,j}$ in \eqref{eq:newrepresentation} are described by
    \begin{align}\label{kappa_Phi}
        \frac{\kappa_{m,j}}{1-\beta} = \delta_t^\alpha \Phi^{j}(t_m)\qquad \forall\, 0\leq j\leq m\leq M,
    \end{align}
where $\delta_t^\alpha \Phi^{j}(t_m)$ $\forall\, k$ is understood as $\pt_t^\alpha (\Pi\Phi^{j})(t_m)$.
\end{lemma}

Next, we formulate sufficient conditions for \eqref{eq:newrepresentation} and \eqref{eq:condition_kappa}
in terms of standard hat-function basis $\{\phi^{j}(t_{k})\}_{j=0}^{M}$ (see Fig.\,\ref{fig_basis}\,(right)) for functions in $\mathbb{R}^{M+1}$ associated with the mesh $\{t_{k}\}_{k=0}^{M}$.

\begin{lemma}[{\citep[Lemma~2.3]{NataliaL2}}]\label{lem_suff}
Given a general temporal grid $\{t_j\}$, the representation \eqref{eq:newrepresentation} of the discrete operator $\delta_t^\alpha$ satisfies \eqref{eq:condition_kappa} if
    \begin{subequations}\label{eq:newcondition}
        \begin{align}
            & \delta_{t}^{\alpha} \phi^{m}(t_m)>0&&\forall\,m\ge 1,\label{eq:newcondition0}\\
            &\delta_{t}^{\alpha}\phi^{m-1}(t_{m}) + \beta \delta_{t}^{\alpha} \phi^{m}(t_m)\le 0
            &&\forall\,m\ge 2,\label{eq:newcondition1}\\
            &\delta_{t}^{\alpha}\phi^{m-2}(t_{m}) + \beta \bigl[\delta_{t}^{\alpha} \phi^{m-1}(t_m)+ \beta\delta_{t}^{\alpha}\phi^{m}(t_m)\bigr]\leq 0
            &&\forall\,m\ge 3,\label{eq:newcondition2}\\
            &\delta_{t}^{\alpha}\phi^{j}(t_{m})\le 0&&\forall\, 1\le j\leq m-3.\label{eq:newcondition3}
        \end{align}
    \end{subequations}
\end{lemma}
\begin{proof}
	While the proof is similar to the one in \citep{NataliaL2}, we sketch it since, by contrast to \citep{NataliaL2}, the present lemma includes an additional condition \eqref{eq:newcondition2} and is formulated under no restrictions on the grid.
First, note that
    \beq\label{Phi_phi}
        \Phi^m(t_k)=\phi^m(t_k)\;\;\forall\,k\le m,\qquad
        \Phi^j(t_k)=\phi^j(t_k)+\beta\,\Phi^{j+1}(t_k)\;\;\forall\,j,k\ge 0.
    \eeq
Combining this with \eqref{eq:Phi}, one concludes that \eqref{eq:newcondition0}, \eqref{eq:newcondition1}, and \eqref{eq:newcondition2} are respectively equivalent to
 $\kappa_{m,m}>0$, $\kappa_{m,m-1}\le 0$, and $\kappa_{m,m-2}\le 0$.
Finally, $\kappa_{m,j}\le 0$ for any $j\le m-3$ is proved by induction, as \eqref{eq:Phi} combined with \eqref{Phi_phi} yields $\kappa_{m,j}=(1-\beta)\delta_t^\alpha\phi^j(t_m)+\beta\kappa_{m,j+1}$, so, in view of  \eqref{eq:newcondition3}, $\kappa_{m,j+1}\le 0$ implies $\kappa_{m,j}\le 0$ for any $j\le m-3$.
\end{proof}

Let $\tau_{j}:=t_{j}-t_{j-1}$ for $m = 1,\dots,M$.
Throughout the remainder of this section, it will be assumed $\forall\,j\ge1$ that, for some fixed constant $\eps\ge 0$,
    \beq\label{rho_cond}
        1\le \rho_{j+1}:=\frac{\tau_{j+1}}{\tau_{j}}\le1+\eps
        \quad\Rightarrow\quad
        0\le \frac{\tau_{j+1}-\tau_j}{\tau_{j+1}+\tau_j}\le \frac{\eps}{2+\eps},\quad
        \frac12\le \frac{\tau_{j+1}}{\tau_{j+1}+\tau_j}\le \frac{1+\eps}{2+\eps}.
    \eeq
Later we shall show that this condition may be relaxed, while the main result of this section will be presented under condition~\eqref{rho_cond} in the following theorem, to the proof of which we devote the remainder of this section.

\begin{theorem}\label{theo_beta_repr}
There exists a sufficiently small positive constant $\eps$, independent of $\alpha$, such that,
if the temporal grid satisfies \eqref{rho_cond} $\forall\, j\ge 1$, then the representation \eqref{eq:newrepresentation} of the discrete operator $\delta_t^\alpha$ of \eqref{eq:defdelta} satisfies \eqref{eq:condition_kappa} with $\beta:=1-2^{-\alpha}$.
\end{theorem}
\begin{proof}
 In view of Lemma~\ref{lem_suff}, it suffices to check that conditions \eqref{eq:newcondition} are satisfied.
Condition \eqref{eq:newcondition3} is established in Lemma~\ref{lem_cond3}, while
conditions \eqref{eq:newcondition0}--\eqref{eq:newcondition2} in Lemma~\ref{lem_cond_ab} below.
\end{proof}

\subsection{Preliminary evaluations for $\delta_t^\alpha\phi^j(t_m)$}\label{ssec_prelim}

To check conditions \eqref{eq:newcondition}, we first need to evaluate $\delta_t^\alpha\phi^j(t_m)$ using the definition \eqref{eq:defdelta} of $\delta_t^\alpha$.

Consider the continuous piecewise-linear function $\psi^{j}$ with nodal values $\psi^{j}(t_{k})= 1$ for $k<j$ and $\psi^{j}(t_{k})= 0$ for $k\geq j$, i.e.
${\psi}^{j}(t)$ is $(t_{j}-t)/\tau_{j}$ for $t\in[t_{j-1},t_{j}]$, and, otherwise, it is $1$ for $t\le t_{j-1}$ and $0$ for $t\ge t_j$
(while $\psi^0(t)=0$ $\forall\,t\ge 0$); see Fig.\,\ref{fig_basis}\,(right).
Then
    \beq\label{psi_phi}
        \phi^{j}(t)= {\psi}^{j+1}(t)- {\psi}^{j}(t)
        \quad\Rightarrow\quad
        \delta_{t}^{\alpha}\phi^{j}(t_{m})= \delta_{t}^{\alpha}{\psi}^{j+1}(t_{m})- \delta_{t}^{\alpha}{\psi}^{j}(t_{m})\quad \forall\,j\ge 0.
    \eeq
Next, recalling the definition of $\Pi$ in \eqref{eq:defdelta}, a calculation shows that
    \[
        \Pi\psi^{j}(t)=\psi^{j}(t)-B^j(t-t_j),\qquad
        B^j(s):=
                \left\{
                \begin{array}{cl}
                \displaystyle
                    \frac{s(\tau_{j+1}-s)}{\tau_j(\tau_j+\tau_{j+1})}& \mbox{for}\ s\in(0,\tau_{j+1}),\;j\ge1,\\[0.3cm]
                \displaystyle
                    \frac{s(\tau_{j}+s)}{\tau_j(\tau_{j-1}+\tau_j)}&  \mbox{for}\ s\in(-\tau_j,0),\;j\ge 2,\\[0.5cm]
                0&\mbox{otherwise}.
                \end{array}
                \right.
    \]
For $j=1$, the above implies that $B^1(s):=0$ for $s\in(-\tau_1,0)$, which is equivalent to $B^1(t-t_1)=0$ for $t\in(0,t_1)$.
Hence, by \eqref{eq:defdelta},
    \beq\label{psi_B}
        \delta_t^\alpha \psi^j(t_m)
        =\pt_t^\alpha (\Pi\psi^j)(t_m)
        =\pt_t^\alpha \bigl[\psi^j(t)-B^j(t-t_j)\bigr]\Big|_{t=t_m}
        \quad\forall\,j\ge 1.
    \eeq
Here, to evaluate $\delta_t^\alpha \psi^j(t_m)$, note that for any $t>t_j$ with $j\ge 1$, by \eqref{eq:def-Caputoderivative}, one has
    \beq\label{pt_psi}
        \Gamma(1-\alpha)\,\pt_t^\alpha \psi^j(t)
        =-\int_{t_{j-1}}^{t_{j}}\!\! (t-s)^{-\alpha} \,\tau_j^{-1}  \, \dd s
        =\frac{(t-t_j)^{1-\alpha}-(t-t_{j-1})^{1-\alpha}}{(1-\alpha)\,\tau_{j}}.
    \eeq
Next, we need to evaluate $\pt_t^\alpha B^j(t-t_j)\big|_{t=t_m}$ in \eqref{psi_B}, for which we recall that
$B^j$ has support on $(-\tau_j,\tau_{j+1})$ for $j\ge2$ and $(0,\tau_2)$ for $j=1$.
Thus, it is convenient to evaluate
$\pt_{t;-\infty}^\alpha B^j(t)$ for any $t>\tau_{j+1}$, where $\pt_{t;-\infty}^\alpha$ is defined similarly to
\eqref{eq:def-Caputoderivative}, only  with $\int_0^t$ replaced by $\int_{-\infty}^t$.
For this modification of \eqref{eq:def-Caputoderivative}, an integration by parts yields
    \beq\label{delta_BB}
        -\Gamma(1-\alpha)\,\pt_{t;-\infty}^\alpha B^j(t)=\alpha\int_{-\tau_{j}}^{\tau_{j+1}}\!\!(t-s)^{-\alpha-1}\,B^j(s) \,  \dd s
        =
        \frac{b_{j}(t)-b_{j-1}(t+\tau_j)}{\tau_j},
    \eeq
where $t\ge\tau_{j+1}$ and
    \beq\label{b_j_def}
        b_{j}(t):=\alpha\int_0^{\tau_{j+1}}\frac{s(\tau_{j+1}-s)}{
        \tau_{j}+\tau_{j+1} }\,(t-s)^{-\alpha-1}\, \dd s
        \quad\forall\,j\ge 1,
        \qquad
        b_0(t):=0.
    \eeq

Now, combining  the above findings
 \eqref{psi_B}, \eqref{pt_psi}, and \eqref{delta_BB}
(using $t:=t_m-t_j$ and $t+\tau_j:=(t_m-t_j)+\tau_j=t_m-t_{j-1}$ in \eqref{delta_BB}) yields, $\forall\,1\le j\le m-1$,
    \beq\label{delta_psi}
        \Gamma(1-\alpha)\,\delta_t^\alpha \psi^j(t_m)=
        \frac{(t_m-t_j)^{1-\alpha}-(t_m-t_{j-1})^{1-\alpha}}{(1-\alpha)\,\tau_{j}}
        +\frac{b_{j}(t_m-t_{j})-b_{j-1}(t_m-t_{j-1})}{\tau_j},
    \eeq
or, equivalently, $\forall\,1\le j\le m-1$ one has
    \beq\label{nu_def}
        \Gamma(1-\alpha)\,\delta_t^\alpha \psi^j(t_m)=\frac{v_j(t_m-t_{j})-v_{j-1}(t_m-t_{j-1})}{\tau_{j}}\;\;\forall\,j\ge 1,\quad
        v_j(t):=\frac{t^{1-\alpha}}{1-\alpha}+b_j(t).
    \eeq

Finally, we substitute the above in the representation of $\delta_{t}^{\alpha}\phi^{j}(t_{m})$ in \eqref{psi_phi}
$\forall\,1\le j\le m-2$ as follows:
    \begin{align}\label{delta_phi_nu_a}
        \Gamma(1-\alpha)\,\delta_t^\alpha \phi^j(t_m)&=
        \frac{v_{j+1}(t_m-t_{j+1})-v_{j}(t_m-t_{j})}{\tau_{j+1}}
        -\frac{v_j(t_m-t_{j})-v_{j-1}(t_m-t_{j-1})}{\tau_{j}}.
        \\
\intertext{This immediately reduces to the second derivative of $v_j(t)$ with and added residual term $R^j$:}
        \label{delta_phi_nu}
        \Gamma(1-\alpha)\,\delta_t^\alpha \phi^j(t_m)&=\frac{\tau_j+\tau_{j+1}}{2}\, v_j''(t_m-t^*_{j})+R_j\quad\mbox{for some~~}t_j^*\in(t_{j-1},t_{j+1}),
        \\
\intertext{$\forall\,1\le j\le m-2$, where}
        \label{delta_phi_R}
        R_j&:=\frac{(b_{j+1}-b_j)(t_m-t_{j+1})}{\tau_{j+1}}
        -\frac{(b_j-b_{j-1})(t_m-t_{j-1})}{\tau_{j}}.
    \end{align}

\begin{remark}\label{rem_phi_m}
Setting $b_j(0):=0$ $\forall\,j$ (so far $b_j(t)$ was defined for $t\ge\tau_{j+1}$) allows for \eqref{delta_psi} and \eqref{nu_def} become valid $\forall\,j\le m$, including $\delta_t^\alpha\psi^m(t_m)$, while \eqref{delta_phi_nu_a} will be valid $\forall\,1\le j\le m-1$, including $\delta_t^\alpha\phi^{m-1}(t_m)$.
\end{remark}

To estimate $\delta_t^\alpha \phi^j(t_m)$ $\forall\,j\le m-3$, it will be useful to have ready estimates for both $b_j(t)$ of \eqref{b_j_def} and $b_j''(t)$, for $t\ge 2\tau_j$.
For these functions, a calculation shows that
    \beq\label{b_j_I_beta}
        b_j(t)=\alpha\, t^{-\alpha-1}\frac{\tau_{j+1}^3}{
        \tau_{j}+\tau_{j+1} } I_{\alpha+1}(\tau_{j+1}/t),\;\;\;
        b''_j(t)=\alpha(\alpha+1)(\alpha+2)\,t^{-\alpha-3}\frac{\tau_{j+1}^3}{
        \tau_{j}+\tau_{j+1} }I_{\alpha+3}(\tau_{j+1}/t),
    \eeq
where
    \beq\label{I_beta_def}
        I_{\beta}(\hat\tau):=\int_0^{1}s(1-s)\,(1-s\hat \tau)^{-\beta}\, \dd s
        \;\;\Rightarrow\;\;
        I_{\beta}(\tau_{j+1}/t)=\int_0^{\tau_{j+1}}\frac{s(\tau_{j+1}-s)}{\tau_{j+1}^{3}}\,(1-s/t)^{-\beta}\, \dd s.
    \eeq
For $b_j$, the above is straightforward from \eqref{b_j_def}, while for $b_j''$, we first use $\frac{ \dd  ^2}{ \dd t^2}(t-s)^{-\alpha-1}=(\alpha+1)(\alpha+2)(t-s)^{-\alpha-3}$,
and then represent $b_j''$ in terms of $I_{\alpha+3}$ similarly to $b_j$.
To estimate $I_{\beta}$, note that $(1-s\hat \tau)^{-\beta}$ is increasing in $\hat \tau$ (as $\beta>0$)
and is also increasing in $\beta$, so
    \beq\label{I_beta}
        I_{\beta}(\tau_{j+1}/t)\le I_{\beta}({\textstyle\frac12})=:\bar I_{\beta}\;\;\forall\,t\ge2\tau_{j+1},
        \quad
        \bar I_{\alpha+1}<\bar I_{2}<{\textstyle\frac13},\quad
        \bar I_{\alpha+2}<\bar I_{3}<{\textstyle\frac12},\quad
        \bar I_{\alpha+3}<\bar I_{4}={\textstyle\frac23}.
    \eeq
Here $\bar I_2=12\ln2-8<\frac13$, as well as
$\bar I_3=6-8\ln2<\frac12$
and $\bar I_4=\frac23$, (as well as a more general $\bar I_\beta$) can be evaluated by twice integrating by parts.
Another useful bound,
    \beq\label{I_beta_diff}
        0\le I_{\alpha+1}(\tau_{j+1}/t)-I_{\alpha+1}(\tau_{j}/t)\le\frac{\tau_{j+1}+\tau_{j}}{\tau_{j+1}}\,\frac{3\eps}{4(2+\eps)}\qquad \forall\,t\ge2\tau_{j+1},
    \eeq
follows from $0\le I'_\beta(\hat\tau)\le \beta I_{\beta+1}(\hat\tau)< 3\cdot \frac12 =\frac32$, for the case $\beta:=\alpha+1$ (where we also used
$I_{\alpha+2}(\hat\tau)\le \bar I_{\alpha+2}<\frac12$, in view of \eqref{I_beta}).
Additionally, for \eqref{I_beta_diff} we used
$\frac{\tau_{j+1}-\tau_{j}}{t}\le \frac{\tau_{j+1}+\tau_{j}}{\tau_{j+1}}\,\frac\eps{2(2+\eps)}$, which follows from
$\tau_{j+1}-\tau_{j}\le (\tau_{j+1}+\tau_{j})\,\frac\eps{2+\eps}$, in view of \eqref{rho_cond},
combined with $t\ge2\tau_{j+1}$.

\subsection{Condition \eqref{eq:newcondition3}}\label{se:inv-mono-condc}

Now we are prepared to establish \eqref{eq:newcondition}, and we start with \eqref{eq:newcondition3}.

\begin{lemma}[condition \eqref{eq:newcondition3}]\label{lem_cond3}
There exists a sufficiently small positive constant $\eps$, independent of $\alpha$, such that,
under condition \eqref{rho_cond} $\forall\, j\ge 1$, one enjoys
$\delta_{t}^{\alpha}\phi^{j}(t_{m})\le 0$ $\forall\,1\le j\leq m-3$.
\end{lemma}

\begin{proof}
 To estimate $\delta_{t}^{\alpha}\phi^{j}(t_{m})$ for $1\le j\le m-3$,
we shall rely on a few preliminary results from section~\ref{ssec_prelim}.
Recalling \eqref{delta_phi_nu}, we need to first estimate $v_j''(t)=-\alpha t^{-\alpha-1}+b_j''(t)$ (which readily follows from the definition of $v_j$ in \eqref{nu_def})
for $t\ge  2\tau_{j+1}$.
For this, by \eqref{b_j_I_beta}  combined with
\eqref{I_beta}, 
so $(\alpha+1)(\alpha+2)\, I_{\alpha+3}\le 2\cdot 3\cdot \frac23$ in \eqref{b_j_I_beta}, one gets
    \begin{align}
        -\alpha^{-1}\, v''_j(t)&\ge t^{-\alpha-1}-2\cdot 3\cdot t^{-\alpha-3}\frac{\tau_{j+1}^3}{
        \tau_{j}+\tau_{j+1} }\cdot \frac23
        \notag\\
        \notag
        &\ge
        t^{-\alpha-1}\left[
        1-6\cdot\frac{\tau_{j+1}}{
        4(\tau_{j}+\tau_{j+1}) }\cdot\frac23
        \right]
        \\
        &\ge
        \frac{t^{-\alpha-1}}{2+\eps}
        \qquad\forall\,t\ge2\tau_{j+1},
        \label{v_der_bound}
    \end{align}
where for the second line we  used
$t^{-2}\tau_{j+1}^2\le\frac14$, and for the final line
$1-\frac{\tau_{j+1}}{\tau_{j+1}+\tau_j}\ge \frac{1}{2+\eps}$,  in view of \eqref{rho_cond}.

The above bound \eqref{v_der_bound} for $v_j''(t)$ is now used in \eqref{delta_phi_nu} with $t:=t_m-t^*_{j}$ so $t\ge t_{j+3}-t_{j+1}\ge 2\tau_{j+1}$
(i.e. \eqref{v_der_bound} is applicable),
while $t\le t_m-t_{j-1}$, so $t^{-\alpha-1}\ge (t_m-t_{j-1})^{-\alpha-1}$ in \eqref{v_der_bound},
so
    \beq\label{delta_phi_final}
        -\alpha^{-1}\,\Gamma(1-\alpha)\,\delta_t^\alpha \phi^j(t_m)
        \ge
        \frac{\tau_j+\tau_{j+1}}{2}\,\frac{(t_m-t_{j-1})^{-\alpha-1}}{2+\eps}-\alpha^{-1}\,R_j\qquad
        \forall\,j\ge1.
    \eeq
It remains to estimate $R_j$ defined in \eqref{delta_phi_R},
where both arguments $t_m-{t_{j\pm 1}}\ge t_{j+3}-t_{j+1}\ge 2\tau_{j+1}$.
For $b_j$ we shall employ the representation from \eqref{b_j_I_beta},
which yields
$(b_j-b_{j-1})(t)=\alpha\,S_j(t)\, t^{-\alpha-1}$, where
    $$
        S_j(t):=
        \sigma_{j+1}\, I_{\alpha+1}(\tau_{j+1}/t)-\sigma_{j}\, I_{\alpha+1}(\tau_{j}/t),
        \qquad
        \sigma_{j+1}:=\frac{\tau_{j+1}^3}{\tau_{j}+\tau_{j+1} },
    $$
with one exception $S_1(t):=\sigma_{2}\, I_{\alpha+1}(\tau_{2}/t)>0$ (in view of $b_0=0$).
Hence $R_j$ defined in \eqref{delta_phi_R} can be rewritten as
    $$
        \alpha^{-1}\,R_j=\frac{S_{j+1}(t_m-t_{j+1})}{\tau_{j+1}}\,(t_m-t_{j+1})^{-\alpha-1}
        -\frac{S_j(t_m-t_{j-1})}{\tau_{j}}\,(t_m-t_{j-1})^{-\alpha-1}.
    $$
Note that $|S_j|$  is bounded by $|\sigma_{j+1}-\sigma_{j}|\, \bar I_{\alpha+1}\le |\sigma_{j+1}-\sigma_{j}|\cdot\frac13$ (in view of~\eqref{I_beta}),
summed with $\sigma_{j+1}$ multiplying the bound from \eqref{I_beta_diff}. Hence, one gets
    $$
        \frac{|S_j(t)|}{\tau^2_j}\le
        \frac{|\sigma_{j+1}-\sigma_{j}|}{3\tau_j^2}+
        \frac{\tau_{j+1}^2}{\tau_j^2}\,\frac{3\eps}{4(2+\eps)}
        \le
        \frac{\eps(\frac54+\eps)}{3}+
        \frac{3\eps(1+\eps)^2}{4(2+\eps)}=: C_\eps\quad
        \forall\,j\ge2.
    $$
Here $\tau_{j+1}^2/{\tau_j^2}$ was immediately bounded by $(1+\eps)^2$,
while we also used $\tau_{j}^{-2}|\sigma_{j+1}-\sigma_{j}|\le \eps(\frac54+\eps)$, both in view of \eqref{rho_cond}.
To get the latter, note that
$0\le \tau_{j}^{-2}\tau_{j+1}\frac{\tau_{j+1}^2-\tau_j^2}{\tau_{j}+\tau_{j+1}}\le \eps(1+\eps)$,
while to estimate the remaining difference between $\frac{\tau_{j+1}}{\tau_{j}+\tau_{j+1} }$ and $\frac{\tau_{j}}{\tau_{j-1}+\tau_{j} }$, we recall that both these quantities
are in $[\frac12,\frac{1+\eps}{2+\eps}]$, so their difference is bounded by $\frac{1+\eps}{2+\eps}-\frac12\le \frac14\eps$. Adding $\frac14\eps$ to $\eps(1+\eps)$, indeed, gives the desired upper bound $\eps(\frac54+\eps)$ on $\tau_{j}^{-2}|\sigma_{j+1}-\sigma_{j}|$.
Now, the above bound on $|S_j|/\tau_j^2$ is established, one immediately gets
\begin{align}
\alpha^{-1}\,R_j &\le   C_\eps\,\bigl[{\tau_{j+1}}\,(t_m-t_{j+1})^{-\alpha-1}+
{\tau_{j}}\,(t_m-t_{j-1})^{-\alpha-1}\bigr]
\notag\\
\label{R_j_final}
&\le  C_\eps\,\bigl[4\,{\tau_{j+1}}+
{\tau_{j}}\bigr]\,(t_m-t_{j-1})^{-\alpha-1}.
\end{align}
Note that this bound is valid  $\forall\,j\ge 1$, including the case $j=1$, when we simply used $S_1\ge 0$.
Here we also used
$(t_m-t_{j-1})/(t_m-t_{j+1})\le 1+(t_{j+1}-t_{j-1})/(t_{j+3}-t_{j+1})\le 2$.
Finally, combining \eqref{R_j_final} (in which the constant $C_\eps$ is independent of $\alpha$) with \eqref{delta_phi_final}, for a sufficiently small $\eps$
(such that $(2+\eps)C_\eps\le \frac18$, which is achievable since $C_{\eps}|_{\eps=0}=0$ and $C_{\eps}$ is increasing with respect to $\eps$), one immediately gets the desired $\delta_t^\alpha \phi^j(t_m)\le 0$ for $1\le j\le m-3$.
\end{proof}

\begin{remark}[uniform temporal mesh]
The proof of Lemma~\ref{lem_cond3} somewhat simplifies if $\tau_j=\tau$ $\forall\,j\ge 1$ (which corresponds to $\eps:=0$ in \eqref{rho_cond}).
To be more precise, the estimation of the residual $R^j$ in \eqref{delta_phi_final} becomes unnecessary as
$R^j=0$ $\forall\,j\ge 2$
(in view of $b_{j\pm 1}=b_j$ in
\eqref{delta_phi_R}), while $R^1=-b_1(t_m)/\tau_1<0$ (in view of $b_0=0$).
Additionally, $\tau_{j+1}/(\tau_j+\tau_{j+1})$ simplifies to $\frac12$, which corresponds to
 $\eps=0$ in \eqref{v_der_bound}.
\end{remark}

\subsection{Conditions \eqref{eq:newcondition0}, \eqref{eq:newcondition1}, and \eqref{eq:newcondition2}}

To complete the proof of Theorem~\ref{theo_beta_repr}, in the following Lemma~\ref{lem_cond_ab} we shall establish an equivalent, appropriately normalized, version of conditions \eqref{eq:newcondition0}--\eqref{eq:newcondition2}, rewritten in terms of
    \begin{subequations}\label{AB_def}
        \begin{align}\label{AB_def_a}
            \A_m&:= \tau_{m}^{\alpha}\,\Gamma(2-\alpha)\, \delta_t^\alpha \phi^m(t_m), \\
            \label{AB_def_b}
            -\B_m&:= \tau_{m}^{\alpha}\,\Gamma(2-\alpha)\, \delta_t^\alpha \phi^{m-1}(t_m), \\
            \label{AB_def_f}
            \D_m&:= \tau_{m}^{\alpha}\,\Gamma(2-\alpha)\,
             \delta_t^\alpha \phi^{m-2}(t_m),
         \end{align}
     \end{subequations}
 where we recall that $\Gamma(2-\alpha)=(1-\alpha)\,\Gamma(1-\alpha)$.

\begin{lemma}[conditions \eqref{eq:newcondition0}, \eqref{eq:newcondition1}, and \eqref{eq:newcondition2}]\label{lem_cond_ab}
Set $\beta:=1-2^{-\alpha}$.
There exists a sufficiently small positive constant $\eps$, independent of $\alpha$, such that,
under condition \eqref{rho_cond}, one enjoys
    \beq\label{ABD_cond}
        \A_m>0\quad\forall\, m\ge1, \qquad
        \B_m-\beta\A_m>0\quad\forall\, m\ge2, \qquad
        \D_m\le \beta(\B_m-\beta\A_m)\quad\forall\, m\ge3.
    \eeq
\end{lemma}
\begin{proof}
We shall again rely on a few results from section~\ref{ssec_prelim}.
%
For any fixed $m\ge 1$, it is convenient to define normalized versions of $v_j$ of \eqref{nu_def} and $b_j$ of \eqref{b_j_I_beta} as follows:
    \begin{align}
        \label{bar_v_def}
        \bar v_j(t)&:=\tau_{m}^{\alpha-1} (1-\alpha)\,v_j(t)=
        (t/\tau_m)^{1-\alpha}+\bar b_j(t),
        \\
        \label{bar_b_def}
        \bar b_j(t)&:=\tau_{m}^{\alpha-1}(1-\alpha)\,b_j(t)
        =
        \tau_{m}^{\alpha-1}\alpha(1-\alpha)\,t^{-\alpha-1}\,\tau_{j+1}^2\,\frac{\tau_{j+1}}{
        \tau_{j}+\tau_{j+1} } \,I_{\alpha+1}(\tau_{j+1}/t),
        \\
        \intertext{where, in agreement with special cases of $b_j(t)$, we set $\bar b_0(t):=0$, as well as $\bar b_j(0):=0$.
        We shall also employ related quantities}
        \hat b_{j}&:=\bar b_{j}(t_m-t_{j}),
        \qquad
        \hat v_{j}:=\bar v_{j}(t_m-t_{j})=\bigl[(t_m-t_{j})/\tau_m\bigr]^{1-\alpha}+\hat b_{j}.
        \label{hat_b_v}
    \end{align}

It is useful to immediately compare $\hat b_{m-2}$ and $\hat b_{m-3}$.
For this, we first recall that  the function $I_{\alpha+1}(\cdot)$ of \eqref{I_beta_def} is increasing,
while its argument
$\tau_{j+1}/t$
in the definition \eqref{bar_b_def} of $\bar b_j(t)$
becomes $\tau_{m-1}/(t_m-t_{m-2})$
in $\hat b_{m-2}$ and
and a smaller value $\tau_{m-2}/(t_m-t_{m-3})$ in $\hat b_{m-3}$
(since $\tau_{m-2}\le \tau_{m-1}$, by \eqref{rho_cond}, and $t_m-t_{m-3}>t_m-t_{m-2}$).
Additionally, note that the ratio of the ingredients $\tau_{j+1}/(\tau_j+\tau_{j+1})$ in $\hat b_{m-2}$ and $\hat b_{m-3}$ (respectively with $j=m-2$ and $j=m-3$)
is  $\ge\frac{2+\eps}{2(1+\eps)}$, in view of the final relationship in \eqref{rho_cond}.
Hence,
    \beq\label{b_hat_23}
        \frac{\hat b_{m-2}}{\hat b_{m-3}}
        >
        \Bigl(\frac{t_m-t_{m-3}}{t_m-t_{m-2}}\Bigr)^{\alpha+1}
        \cdot\Bigl(\frac{\tau_{m-1}}{\tau_{m-2}}\Bigr)^{2}\cdot
        \frac{2+\eps}{2(1+\eps)}
        \ge
        \Bigl(1+\frac{1}{(1+\eps)(2+\eps)}\Bigr)\cdot
        1\cdot \frac{2+\eps}{2(1+\eps)}
        \ge 1,
    \eeq
if $\eps$ is sufficiently small (to be more precise, $\eps\le \frac12(\sqrt{5}-1)\approx 0.618$ suffices for this bound).
Note that in \eqref{b_hat_23} we also used
 $[(t_m-t_{m-3})/(t_m-t_{m-2})]^{\alpha+1}\ge (t_m-t_{m-3})/(t_m-t_{m-2})=1+\tau_{m-2}/(\tau_{m-1}+\tau_m)$,
where $\tau_{m-2}\ge \tau_{m-1}/(1+\eps)$, while $\tau_{m-1}/(\tau_{m-1}+\tau_m)\ge 1/(2+\eps)$, both by~\eqref{rho_cond}.
Note also that while  \eqref{b_hat_23} is obtained $\forall\,m>3$,
it also remains valid for $m=3$
(as then $\hat b_{m-3}=\hat b_{0}=0$).

Next, we  evaluate the quantities $\A_m$, $\B_m$, and $\D_m$  (defined in \eqref{AB_def}), which appear in~\eqref{ABD_cond}.
At this stage, it is useful to recall the definitions of $\hat v_j$ and $\hat b_j$ in \eqref{hat_b_v}.
For $\A_m$, note that $\phi^m(t)=1-\psi^m(t)$ for $t\le t_m$ (see Section~\ref{ssec_prelim}), so  \eqref{nu_def} for $\delta_t^\alpha\psi^m(t_m)$
(valid in view of Remark~\ref{rem_phi_m}),
  multiplied by $\tau_m^\alpha\,(1-\alpha)$,
yields
    \begin{align}\label{AB_def_aa}
        \A_m&=
        -\tau_{m}^{\alpha}\,\Gamma(2-\alpha)\,
         \delta_t^\alpha \psi^m(t_m)
         =\hat v_{m-1}-0=1+\hat b_{m-1}>0\quad\forall\,m\ge1,
         \\
\intertext{where we used $\hat v_{m}=\bar v_{m}(0)=0$.
This includes the case $\A_1=\hat v_{0}=1$.
Similarly, for $\B_m$, we multiply \eqref{delta_phi_nu_a}
for $\delta_t^\alpha\phi^{m-1}(t_m)$ (valid in view of Remark~\ref{rem_phi_m})
by $\tau_m^\alpha\,(1-\alpha)$, and then use  $\tau_m/\tau_m=1$ and $\tau_m/\tau_{m-1}=\rho_m$ to get}
        \notag
        -\B_m&=
        1\cdot \bigl [0-\hat v_{m-1}\bigr]-\rho_m\,\bigl [\hat v_{m-1}-\hat v_{m-2}\bigr]\quad\forall\,m\ge2.
\intertext{Recalling that $\beta=1-2^{-\alpha}$, one now gets}
        \notag
        \B_m-\beta \A_m &=
        (2^{-\alpha}+\rho_m)\,\hat v_{m-1}-\rho_m\,\hat v_{m-2}
        \\
        \notag
        &=
        (2^{-\alpha}+\rho_m)\,\bigl[1+\hat b_{m-1}\bigr]-\rho_m\,\bigl[\bigl(1+\rho_m^{-1}\bigr)^{1-\alpha}+\hat b_{m-2}\bigr],
        \\
\intertext{where $(t_m-t_j)/\tau_m$ in $\hat v_j$ became 1 for $j=m-1$ and $1+\rho_m^{-1}$ for $j=m-2$.
Hence, by  rearranging a few terms, one gets}
         \label{AbetaB}
        \B_m-\beta \A_m  &=
        g(\alpha)
        +(2^{-\alpha}+\rho_m)\,\hat b_{m-1}
        -\rho_m\,\hat b_{m-2}\quad\forall\,m\ge2,
        \\
\intertext{where}
        \label{g_def}
        g(\alpha)
        &:=2^{-\alpha}+\rho_m-\rho_m\,\bigl(1+\rho_m^{-1}\bigr)^{1-\alpha},\qquad
        {\textstyle\frac{1}{2}}\alpha\le g(\alpha)\le {\textstyle\frac{1}{2}}\;\;\forall\alpha\in(0,1).
    \end{align}
Here the lower bound $g(\alpha)\ge\frac12\alpha$ follows from $g(0)=0$ and $g(1)=\frac12$ combined with $g(\alpha)$ being a concave function of $\alpha$ for any fixed $\rho_m\in[1,2]$.
 For concavity, it suffices to check that $g''(\alpha)\le g''(0)\le 0$ for this range of $\rho_m$; this corresponds to $\eps\le1$ in \eqref{rho_cond}.
 A similar argument (using $g'(\alpha)\ge g'(0)\ge 0$) shows that $g(\alpha)$ is increasing, so $g(\alpha)\le g(1)=\frac12$.

 For $\D_m$ we use \eqref{delta_phi_nu_a}, with $j=m-2\le 1$,
  multiplied by $\tau_m^\alpha\,(1-\alpha)$, which yields
    \begin{align}
        \notag
        \D_m&=
        \tau_m\,\left(\frac{\hat v_{m-1}-\hat v_{m-2}}{\tau_{m-1}}
        -\frac{\hat v_{m-2}-\hat v_{m-3}}{\tau_{m-2}}\right),
         \\
\intertext{and then}
         \notag
        \D_m&=
        \tau_{m}^{\alpha}\,\left(\frac{(t_m-t_{m-1})^{1-\alpha}-(t_m-t_{m-2})^{1-\alpha}}{\tau_{m-1}}
        -\frac{(t_m-t_{m-2})^{1-\alpha}-(t_m-t_{m-3})^{1-\alpha}}{\tau_{m-2}}\right)
        \\ \notag
        &{}+
        \tau_m\,\left(\frac{\hat b_{m-1}-\hat b_{m-2}}{\tau_{m-1}}
        -\frac{\hat b_{m-2}-\hat b_{m-3}}{\tau_{m-2}}\right).
        \\
\intertext{From this, using \eqref{b_hat_23} and also $\frac{\dd ^2}{\dd t^2}(t_m-t)^{1-\alpha}=-\alpha(1-\alpha)\,(t_m-t)^{-\alpha-1}$, one gets the bound}
        \label{D_bound}
        \D_m&\le-\tau_{m}^{\alpha}\,\frac{\tau_{m-2}+\tau_{m-1}}{2}\alpha(1-\alpha)\,(t_m-t_{m-2})^{-\alpha-1}
        +
        \rho_m\,\bigl[\hat b_{m-1}-\hat b_{m-2}\bigr]\quad\forall\,m\ge 3.
        %
    \end{align}
    Note that here, when bounding the central difference approximation of $\frac{\dd ^2}{\dd t^2}(t_m-t)^{1-\alpha}$ about $t=t_{m-2}$, we only
 used the principal term, while the remainder of the associated Taylor series, being negative, is dropped (in view of $\frac{\dd ^k}{\dd t^k}(t_m-t)^{1-\alpha}<0$ for $k\ge 1$).

Now that we have evaluated
$\A_m$, $\B_m$, and $\D_m$, we proceed to checking that both conditions in \eqref{ABD_cond} are satisfied.
First, we shall establish these conditions
for the case of uniform temporal grid, i.e. $\tau_m:=\tau=1/M$.
For this case $\rho_m=\rho_{m-1}=1$, so if $m\ge 3$, the above findings reduce to
    \begin{align*}
        \A_m&= \bar v(\tau),
        \\
        \B_m&= 2\bar v(\tau)- \bar v(2\tau),
        \\
        \D_m&= \bar v(\tau)-2 \bar v(2\tau)+\bar v(3\tau),
    \end{align*}
where, for $k=1,2,3$,
    \beq\label{c_k_def}
        \bar v(k\tau)=k^{1-\alpha}+c_k(\alpha),\quad
        c_k(\alpha):=\frac{\alpha(1-\alpha)}{2\,k^{\alpha+1}}\,I_{\alpha+1}(1/k),
        \quad
        c_1(\alpha)=\frac{\alpha}{2(2-\alpha)}
        \le \alpha\,2^{\alpha-2}.
    \eeq
Here the definition of $\bar v(\cdot)$ is the same as of $\bar v_j(\cdot)$ in \eqref{bar_v_def} (and the latter is now independent of $j$),
while $\bar v(k\tau)=\hat v_{m-k}$ and $\hat b_{m-k}=c_k(\alpha)$ are also independent of $m$ (compare with \eqref{hat_b_v}).
Note also that $c_1(\alpha)$ was evaluated using the definition of $I_{\alpha+1}$ in \eqref{I_beta_def}, and then evaluated using $2-\alpha\ge 2^{1-\alpha}$.
Thus, \eqref{AbetaB} and \eqref{D_bound} simplify to
    \begin{align}
        \label{AbetaB_uni}
        \B_m-\beta\,\A_m&= g(\alpha)+(2^{-\alpha}+1)\,c_1(\alpha)-c_2(\alpha),
        \\
        \intertext{where now $g(\alpha)=1-2^{-\alpha}$,}
        \notag
        \D_m&\le -\alpha(1-\alpha)\,2^{-\alpha-1}+c_1(\alpha)-c_2(\alpha).
    \end{align}
The second desired condition $\B_m-\beta\,\A_m>0$ immediately follows from $c_2(\alpha)<c_1(\alpha)$ (for the latter, recall that the function $I_{\alpha+1}(\cdot)$ is increasing).
For the third desired condition, when comparing $\D_m$ and $\beta(\B_m-\beta\,\A_m)$, note that $c_2(\alpha)>0$ and $\beta\in(0,1)$, so it suffices to prove
the following stronger bound (in which $c_2(\alpha)$ is eliminated):
    \beq\label{b_suff_uni}
        -\alpha(1-\alpha)\,2^{-\alpha-1}+c_1(\alpha)\le \beta\bigl[g(\alpha)+(2^{-\alpha}+1)\,c_1(\alpha)\bigr].
    \eeq
Note  that here  $g(\alpha)=1-2^{-\alpha}\ge\frac12\alpha$ (by \eqref{g_def}), while $\beta(2^{-\alpha}+1)=1-4^{-\alpha}$, so
it suffices to show that
    \beq\label{non_pert}
        -\alpha(1-\alpha)\,2^{-\alpha-1}
        +4^{-\alpha}c_1(\alpha)\le \beta\cdot{\textstyle\frac12}\,\alpha.
        \eeq
        On multiplication by $2^{\alpha+1}\alpha^{-1}$, noting that $\beta\,2^\alpha=2^\alpha-1$, one now gets
        $$
        -(1-\alpha)
        +2^{1-\alpha}\alpha^{-1}c_1(\alpha)\le 2^\alpha-1.
        $$
        Finally, recalling that $c_1(\alpha)
        \le \alpha\,2^{\alpha-2}$, we conclude that the second condition in \eqref{ABD_cond} for any $m\ge 3$ follows from
        \beq\label{non_pert1}
        \alpha
        +{\textstyle\frac12}\le 2^\alpha,
    \eeq
which  is true $\forall\,\alpha\in(0,1)$ (in fact, $2^{\alpha}-\alpha-\frac12> 0.4$ on this interval).

For $m=1$ one has $\A_1=1>0$. For $m=2$ there is no change in $\A_2= 1+c_1(\alpha)>0$,
while \eqref{AbetaB}
includes $\hat b_{m-2}=\hat b_{0}=0$,
which results in a version of
 \eqref{AbetaB_uni}
 for $m=2$ with $c_2(\alpha)$ replaced by $0$.
 Thus $\B_2-\beta\A_2=g(\alpha)+(2^{-\alpha}+1)\,c_1(\alpha)
 >0$.

Along the lines of the above estimation, we now establish \eqref{ABD_cond} in a more intricate case of the temporal mesh under condition \eqref{rho_cond}.
For any $m\ge 3$, recall \eqref{AbetaB} and \eqref{D_bound}. Now, the first bound in \eqref{ABD_cond} follows from $\hat b_{m-1}\ge\hat b_{m-2}$, which
is obtained similarly to \eqref{b_hat_23}. For the second bound in \eqref{ABD_cond}, it suffices to check that
    \begin{align*}
        -\tau_{m}^{\alpha}\,\frac{\tau_{m-2}+\tau_{m-1}}{2}\alpha(1-\alpha)\,(t_m-t_{m-2})^{-\alpha-1}
        &{}+
        \rho_m\,\bigl[\hat b_{m-1}-\hat b_{m-2}\bigr]
        \\&
           \le
        \beta\Bigl[
        g(\alpha)
          +(2^{-\alpha}+\rho_m)\,\hat b_{m-1}
          -\rho_m\,\hat b_{m-2}
        \Bigr].
    \end{align*}
Note that $\hat b_{m-2}>0$ and $\beta\in(0,1)$, so it suffices to prove
the following stronger bound (in which $\hat b_{m-2}$ is eliminated, and which is a version of \eqref{b_suff_uni} for a non-uniform mesh):
    $$
        -\tau_{m}^{\alpha}\,\frac{\tau_{m-2}+\tau_{m-1}}{2}\alpha(1-\alpha)\,(t_m-t_{m-2})^{-\alpha-1}
        +
        \rho_m\,\hat b_{m-1}\le
        \beta\bigl[
        g(\alpha)
          +(2^{-\alpha}+\rho_m)\,\hat b_{m-1}
        \bigr].
    $$
Recall that $g(\alpha)\ge {\textstyle\frac{1}{2}}\alpha$, in view of \eqref{g_def},
and note that  $\rho_m-\beta(2^{-\alpha}+\rho_m)=2^{-\alpha}(\rho_m-\beta)\le \eps\,2^{-\alpha}+4^{-\alpha}$
(by \eqref{rho_cond}). Hence, it suffices to prove that
    \beq\label{pert}
        -\alpha(1-\alpha)\,(1-\eps)\,2^{-\alpha-1}
        +
        \bigl(\eps\,2^{-\alpha}+4^{-\alpha}\bigr)\,\hat b_{m-1}
        \le
        \beta\cdot{\textstyle\frac12}\,\alpha.
    \eeq
Here we also used $\tau_{m}^{\alpha}\,(t_m-t_{m-2})^{-\alpha}\ge 2^{-\alpha}$
(due to $\tau_m\ge \tau_{m-1}$) combined with
$(1+\eps)(\tau_{m-2}+\tau_{m-1})\ge \tau_{m-1}+\tau_{m}=t_m-t_{m-2}$ and $(1+\eps)^{-1}\ge1-\eps$.
We now recall \eqref{hat_b_v}, which yields  $\hat b_{m-1}=\bar b_{m-1}(\tau_m)$, which, combined with \eqref{bar_b_def} and compared with $c_1(\alpha)$,
yields
    \beq\label{b_bound}
        \hat b_{m-1}=\frac{2\tau_m}{\tau_{m-1}+\tau_m}\,c_1(\alpha)\le (1+\eps)\,c_1(\alpha)\qquad \forall\,m\ge2.
    \eeq
Thus, replacing $\hat b_{m-1}$ in \eqref{pert} by $(1+\eps)\,c_1(\alpha)$ yields a stronger condition,
(which is a perturbed version of \eqref{non_pert}, and becomes identical with \eqref{non_pert} if $\eps=0$).
Multiplying the latter by $2^{\alpha+1}\alpha^{-1}$, and  noting that $\beta\,2^\alpha=2^\alpha-1$, one gets
    $$
        -(1-\alpha)\,(1-\eps)
        +
        (1+\eps)\,\bigl(2\eps+2^{1-\alpha}\bigr)\,\alpha^{-1}c_1(\alpha)
        \le
        2^\alpha-1.
    $$
Here
$-(1-\alpha)\,(1-\eps)\le -1+\alpha+\eps$, while
 $\alpha^{-1}c_1(\alpha)\le 2^{\alpha-2}$, so it suffices to check that
    $$
    \alpha + \eps + (1+\eps)\,\bigl(\eps+{\textstyle\frac12}\bigr) \le    2^\alpha.
    $$
The latter is a perturbed version of \eqref{non_pert1}, and is clearly satisfied if $\eps$ is sufficiently small.
Thus, we have established \eqref{ABD_cond} for $m\ge3$.

Finally, for $m=1,2$ one gets $\A_1=1>0$ and  $\A_2= 1+\hat b_{1}>0$, while
\eqref{AbetaB} for $m=2$   includes $\hat b_{m-2}=\hat b_{0}=0$, so
$\B_2 -\beta\A_2\le \A_2= g(\alpha)
  +(2^{-\alpha}+\rho_2)\,\hat b_{1}>0$.
Thus, we have established \eqref{ABD_cond}.
\end{proof}

\subsection{Stronger version of \eqref{eq:newcondition1}}

The error analysis in the parabolic case will require a stronger version
of  \eqref{eq:newcondition1}, or, equivalently, a stronger version of the second
bound in \eqref{ABD_cond} of Lemma~\ref{lem_cond_ab}, which we now establish.

\begin{lemma}[stronger version of \eqref{eq:newcondition1}]\label{lem_cond_aaa}
There exists a sufficiently small $\eps$, independent of $\alpha$, such that
if the temporal grid satisfies \eqref{rho_cond} and $\beta=1-2^{-\alpha}$, then
    \begin{align}
        \label{AB_cond_stronger}
        \frac{\B_m-\beta\A_m}{\A_m}&\ge(1+\eps)\beta\qquad\forall\, m\ge2,
        \\
        \label{AB_cond_stronger2}
        \beta^2(\tau_m/\tau_{m-1})^{\alpha}\A_m\A_{m-1}&\le (\B_m-\beta\A_m)^2\qquad\forall\, m\ge2.
    \end{align}
\end{lemma}

\begin{proof}
First, note that \eqref{AB_cond_stronger} implies \eqref{AB_cond_stronger2}.
Indeed, assuming that \eqref{AB_cond_stronger} is true, one gets
$(\B_m-\beta\A_m)^2\ge (1+\eps)^2\beta^2\A_m^2$, so
it suffices to check that $(\tau_m/\tau_{m-1})^{\alpha}\A_{m-1}/\A_m\le (1+\eps)^2$.
In the latter,
 $\tau_m/\tau_{m-1}\le 1+\eps$, by \eqref{rho_cond}, while $\A_{m-1}/\A_m\le 1+\eps$,
since, by \eqref{AB_def_aa}, $\A_m=1+\hat b_{m-1}$, so
\eqref{b_bound} yields
$\A_m=1+\frac{2\tau_m}{\tau_{m-1}+\tau_m}\,c_1(\alpha)\ge 1+c_1(\alpha)$
and
$\A_{m-1}=1+\frac{2\tau_{m-1}}{\tau_{m-2}+\tau_{m-1}}\,c_1(\alpha)\le 1+(1+\eps)c_1(\alpha)$, with the exception of $\A_{1}=1$.
So we, indeed, get \eqref{AB_cond_stronger2}.


To prove the remaining bound \eqref{AB_cond_stronger}, it suffice to show that $\B_m-2\beta\A_m\ge \eps\beta\A_m$, for which we shall employ
a few bounds from Lemma~\ref{lem_cond_ab}.
We start by checking \eqref{AB_cond_stronger} for the uniform grid, which corresponds to $\eps=0$ in \eqref{rho_cond}.
Recalling that $\A_m=1+c_1(\alpha)$, one gets a version of \eqref{AbetaB_uni} $\forall\,m\ge3$
(while for $m=2$ one should replace $c_2(\alpha)$ by $0$, so the analysis below also works):
    \begin{align*}
        \B_m-2\beta\A_m&
        = \bigl[g(\alpha)-\beta\bigr]
        +\bigl[2^{-\alpha}+1-\beta\bigr]\,c_1(\alpha)-c_2(\alpha)\\
        &
        \ge
        \bigl[2^{1-\alpha}-2^{-\alpha-1}\bigr]\,c_1(\alpha)= 2^{-\alpha-1}\,3c_1(\alpha).
    \end{align*}
Here we used $g(\alpha)=\beta=1-2^{-\alpha}$, so $[2^{-\alpha}+1-\beta]=2^{1-\alpha}$,
and we also used
$c_2(\alpha)/c_1(\alpha)\le  2^{-\alpha-1}$
(the latter follows from \eqref{c_k_def}, where $I_{\alpha+1}(\cdot)$ is increasing,
so $I_{\alpha+1}(1/2)<I_{\alpha+1}(1)$).
Finally, note that
 $3 c_1(\alpha)=\frac{3\alpha}{2(2-\alpha)}\ge 2^\alpha-1=2^\alpha \beta$, while
$\A_m=1+c_1(\alpha)\le \frac32$, so
    \beq\label{aux_ext_lem_uni}
        \B_m-2\beta\A_m\ge {\textstyle\frac12}\beta\ge  {\textstyle\frac13}\beta\A_m.
    \eeq
Thus, for the uniform temporal grid, we have obtained a version of \eqref{AB_cond_stronger}  with $\eps$ replaced by $\frac13$.

Now we shall imitate the above evaluation for a general grid under condition \eqref{rho_cond},
for which we get a version of \eqref{AbetaB} (in view of $\A_m=1+\hat b_{m-1}$, by \eqref{AB_def_aa}):
    \begin{align}\label{aux_ext_lem}
        \B_m- 2\beta\A_m  &=
        \bigl[g(\alpha)-\beta\bigr]
        +\bigl[2^{-\alpha}+\rho_m-\beta\bigr]\,\hat b_{m-1}
        -\rho_m\,\hat b_{m-2}.
    \end{align}
For $m=2$, one has $\hat b_{m-2}=0$, while for $m\ge 3$
 we need a bound on  $\frac{\hat b_{m-1}}{\hat b_{m-2}}$, which is obtained similarly to
the bound on
$\frac{\hat b_{m-2}}{\hat b_{m-3}}$
in
\eqref{b_hat_23}, only instead of  $(t_m-t_{m-3})/(t_m-t_{m-2})$ we now use the bound on
$(t_m-t_{m-2})/(t_m-t_{m-1})=(\tau_m+\tau_{m-1})/\tau_m\ge(2+\eps)/(1+\eps)\ge 2/(1+\eps)$:
      $$
         \frac{\hat b_{m-1}}{\hat b_{m-2}}
        >
        \Bigl(\frac{t_m-t_{m-2}}{t_m-t_{m-1}}\Bigr)^{\alpha+1}
        \cdot\Bigl(\frac{\tau_{m}}{\tau_{m-1}}\Bigr)^{2}\cdot
        \frac{2+\eps}{2(1+\eps)}
        \ge
        \Bigl(\frac{2}{1+\eps}\Bigr)^{\alpha+1}\cdot
        1\cdot\frac{1}{1+\eps}\ge \frac{2^{\alpha+1}}{(1+\eps)^3}.
    $$
Thus
 $\hat b_{m-2}\le  2^{-\alpha-1}(1+\eps)^3\,\hat b_{m-1}$ $\forall\,m\ge 2$, and, if $\eps$ is sufficiently small, we also enjoy $\hat b_{m-1}\ge \hat b_{m-2}$,
which yield
      $$
          \rho_m(\hat b_{m-1}- \hat b_{m-2})\ge \hat b_{m-1}- \hat b_{m-2}
         \ge \bigl[1-2^{-\alpha-1}(1+\eps)^3\bigr]\hat b_{m-1}.
     $$
 Combining this with \eqref{aux_ext_lem} and  $2^{-\alpha}-\beta=2^{1-\alpha}-1$, we arrive at
    \begin{align}\notag
           \B_m- 2\beta\A_m  &\ge
        \bigl[g(\alpha)-\beta\bigr]
        +2^{-\alpha-1}\bigl[4-(1+\eps)^3\bigr]\,\hat b_{m-1}
        \\\label{aux_ext_lem1}
          &\ge
        \bigl[g(\alpha)-\beta\bigr]
        +{\textstyle\frac12}\bigl[1-\eps-\eps^2-{\textstyle\frac13}\eps^3\bigr]\,\beta.
    \end{align}
Here in the second line we also used $\hat b_{m-1}\ge c_1(\alpha)$ (in view of \eqref{b_bound})
and $3 c_1(\alpha)\ge2^\alpha \beta$ (already used for the uniform grid).
Clearly, \eqref{aux_ext_lem1} is a perturbed version of \eqref{aux_ext_lem_uni}, in which, compared to the uniform grid case, the term
$g(\alpha)-\beta\le 0$ does not vanish, so  it has to be estimated.
The lowest bound is obtained by replacing $\rho_m$ by its largest value $1+\eps$ in the definition \eqref{g_def} of $g(\alpha)$, which yields
    \begin{align*}
        \beta-g(\alpha)&=
        \beta-2^{-\alpha}-(1+\eps)+(1+\eps)^\alpha\cdot(2+\eps)^{1-\alpha}\\
        &\le -2^{1-\alpha}-\eps +[1+\alpha \eps]
        \cdot 2^{-\alpha}\,[2+(1-\alpha)\eps]
        \\
        &\le
        \eps\,2^{-\alpha}\,\bigl[1+\alpha-2^{\alpha}+\alpha\eps\bigr],
\intertext{where we used $(1+\eps)^\alpha\le [1+\alpha \eps]$
(since $(1+\eps)^\alpha$, being a concave function of $\eps$,
is below its tangent line $1+\alpha \eps$) and a similar formula
$(2+\eps)^{1-\alpha}\leq 2^{1-\alpha}[1+(1-\alpha)\frac12\eps]=
2^{-\alpha}[2+(1-\alpha)\eps]$.
Finally, note that $0\le 1+\alpha-2^{\alpha}\le \frac12 2^\alpha\beta$ and
$ \alpha \le \frac32(2^\alpha-1)= \frac32 2^\alpha\beta$,
so }
        \beta-g(\alpha) &\le {\textstyle\frac12}\eps\,(1+3\eps)\,\beta.
    \end{align*}
Combining this with \eqref{aux_ext_lem1} yields
    \begin{align*}
        \B_m- 2\beta\A_m
        &\ge
        {\textstyle\frac12}\bigl[1-\eps-\eps^2-{\textstyle\frac13}\eps^3-\eps\,(1+3\eps)\bigr]\,\beta,
    \end{align*}
in which it remains to choose a sufficiently small $\eps$ such that $\frac12[\cdots]\ge\eps$, which completes the proof of \eqref{AB_cond_stronger}.
\end{proof}

\begin{remark}\label{rem_kappa}
The main result of Lemma~\ref{lem_cond_aaa},
\eqref{AB_cond_stronger2}, is equivalent to
$\beta^2\,\kappa_{m-1,m-1}\, \kappa_{m,m}\le |\kappa_{m,m-1}|^2$.
This becomes obvious
once we recall that the left-hand sides in \eqref{eq:newcondition0} and \eqref{eq:newcondition1}
are, respectively, $\kappa_{m,m}$ and $\kappa_{m,m-1}$, and
 the definitions of $\A_m$ and $\B_m$ in, respectively, \eqref{AB_def_a} and \eqref{AB_def_b}.
\end{remark}

\section{Stability properties for the discrete fractional-derivative operator}\label{sec_stability}

Theorem~\ref{theo_beta_repr} immediately implies inverse monotonicity of the discrete operator $\delta_t^\alpha$ in the sense that  if $U^0=0$, the corresponding matrix
(being a product of two M-matrices, as elaborated in \citep[Remark~2.1]{NataliaL2}) is inverse monotone, i.e. all elements of the inverse of this matrix
are non-negative. Thus, Theorem~\ref{theo_beta_repr} immediately implies the following result (see also a similar \citep[Theorem~3.1]{NataliaL2}).

\begin{corollary}[discrete comparison principle]\label{th-discretecomparisonprinciple}
There exists a sufficiently small positive constant $\eps$, independent of $\alpha$, such that,
if the temporal grid satisfies \eqref{rho_cond} $\forall\, j\ge 1$,  then the following statements are true:
\smallskip

(i) If $U^{0}= 0$ and $\delta_{t}^{\alpha}U^{m}\geq 0$ $\forall\, m\geq 1$, then $U^{j}\geq0$ for all $j\geq 0$.
\smallskip

(ii) If for a certain barrier function $\{W^{j}\}_{j=0}^{M}$ one has $|U^{0}|= W^{0}$ and $|\delta_{t}^{\alpha}U^{m}|\leq \delta_{t}^{\alpha}W^{m}$ $\forall\, m\geq 1$, then $|U^{j}|\leq W^{j}$ for all $j\geq 0$.
\end{corollary}

Throughout the remainder of the paper we shall assume that the temporal grid is quasi-graded in the following sense:
    \begin{align}\label{eq:meshstability}
        \tau_{1}\simeq M^{-r},\quad \tau_{j}\simeq \tau_{1}^{1/r}t_{j}^{1-1/r}
        \quad \forall j =1,\dots,M
     \end{align}

 \begin{remark}[graded versus quasi-graded grids]\label{rem_quasi}
 To better understand \eqref{eq:meshstability}, note that dividing the condition on $\tau_j=t_j-t_{j-1}$ by $t_j^{1-1/r}$,
 and then noting that $(t_{j-1}/t_j)^{1-1/r}\le 1$), yields
     $$
        t_j^{1/r}-t_{j-1}^{1/r}
        \le \frac{t_j-t_{j-1}}{t_j^{1-1/r}}
      \simeq \tau_{1}^{1/r} \simeq M^{-1}.
     $$
 The most obvious example of \eqref{eq:meshstability} is
 the standard graded grid $t_{j}=T(j/M)^{r}$. Additionally, a straightforward evaluation shows that
     $$
        1\le \frac{\tau_{j+1}}{\tau_{j}}=\frac{(j+1)^r-j^r}{j^r-(j-1)^r}
        =\frac{(1+1/j)^r-1}{1-(1-1/j)^r}\rightarrow 1^+\quad \mbox{as~}j\rightarrow \infty,
    $$
so for any $\eps>0$, one finds a sufficiently large $K=K(\eps, r)$ such that the graded grid satisfies \eqref{rho_cond} $\forall\,j\ge K$, and, hence, also Theorem~\ref{theo_beta_repr}
and Corollary~\ref{th-discretecomparisonprinciple}.
Furthermore,  one can construct a modified graded grid that satisfies both \eqref{eq:meshstability} and \eqref{rho_cond} $\forall\,j\ge 1$,
as described in \citep[Remark~3.4]{NataliaL2}.
\end{remark}

\begin{theorem}[stability]\label{th-euniformgrid}
Let the temporal mesh satisfy \eqref{eq:meshstability}  for some $r\geq 1$, and also \eqref{rho_cond}
with a sufficiently small $\eps$
$\forall\,j\ge K$ for some fixed $1\le K\lesssim 1$ (i.e. $K$ may be sufficiently large independently of $M$).
  Then,   given $\gamma\in\mathbb{R}$, the following is true for $\{U^{j}\}_{j=0}^{M}$:
     \begin{align}\label{eq:stability}
        \left.
        \begin{array}{c}
           |\delta_{t}^{\alpha}U^{m}| \lesssim (\tau_{1}/t_{m})^{\gamma+1} \;\\[2pt]
            \;\forall m\geq 1,\quad U^{0}=0
        \end{array}\!\!\!
        \right\}
         \;\;\Rightarrow\;\; |U^{m}| \lesssim \mathcal{U}^{m}(\tau_{1};\gamma):=\ell_{\gamma}\, \tau_{1}\,t_{m}^{\alpha-1} (\tau_{1}/t_{m})^{\min\{0,\gamma\}},
    \end{align}
where $\ell_{\gamma}=\ell_{\gamma}(t_{m}):=1+\ln(t_{m}/\tau_{1})$ for $\gamma=0$, and $\ell_{\gamma}:=1$ otherwise.
\end{theorem}

\begin{remark}[no restriction on $r$ in Theorem~\ref{th-euniformgrid}]\label{rem_th-euniformgrid}
A similar stability property was established in
\citep[section 2.2]{NataliaMeng} for the L1 method, \citep[section 4.2]{NataliaMeng} for the Alikhanov scheme, and
in \citep{NataliaL2} for a forward-quadratic L2 method. However, these results were obtained under the restriction that
$ r\le \frac{2-\alpha}{\alpha}$ for the L1 method, and $r\le \frac{3-\alpha}{\alpha}$ for the higher-order methods.
Any restriction on $r\ge 1$ was recently dropped due to a new proof in \citep{NataliaSean},
which we now extend to the backward L2 method.
\end{remark}

The remainder of this section is devoted to the proof of  Theorem~\ref{th-euniformgrid}, which, as we shall see below, reduces to
establishing the following Lemmas~\ref{lem1}, Corollary~\ref{cor_lem1}, and Lemma~\ref{lem2}.

\begin{lemma}\label{lem1}
The statement of Theorem~\ref{th-euniformgrid} holds true if $K=1$.
\end{lemma}

\begin{proof}
 In view of Corollary~\ref{th-discretecomparisonprinciple}, it suffices to
find a barrier function $0\le W(t)\lesssim \tau_{1}^{-1}\mathcal{U}^{m}$ such that $\pt_t^\alpha W(t)\gtrsim t_{m}^{-1}(\tau_{1}/t_{m})^{\gamma}$
and
its truncation error $r_W^m:=(\delta_t^\alpha-\pt_t^\alpha) W(t_m)$ is dominated by
$\frac12\pt_t^\alpha W(t_m)$.
For the L1 method, such $W$ was constructed in the proof of \cite[Lemma 2.7]{NataliaSean}  
%
 as a linear combination of prototype barriers
$$
B(t)=B(t\,;t_n):=\min\bigl\{(t/t_n)\,t_n^{\alpha-1}, t^{\alpha-1} \bigr\},
$$
with various $n\ge p$, where $p\simeq 1$ is a sufficiently large fixed number.
Such prototype barriers are helpful, since an elementary evaluation (see, for example, \citep[(2.14)]{NataliaSean}) yields
an explicit bound
\beq\label{caputo_B}
\pt^\alpha_t B(t)
\ge
c_\alpha\,
 (t_n/t)^{\alpha}\,t^{-1}
\quad\forall\,t>t_n>0, \qquad c_\alpha:=\frac{1-\alpha}{\Gamma(1-\alpha)}.
\eeq
While the construction and estimation of $W$ and $\delta_t^\alpha W$ in \citep{NataliaSean}
are rather intricate,
 for the proof of \citep[Lemma 2.7]{NataliaSean} to apply to our method (or any other inverse-monotone method), it suffices to check that the truncation error
 $r_B^m:=(\delta_t^\alpha-\pt_t^\alpha) B(t_m)$
for our higher-order method enjoys
at least the same bound as the corresponding truncation error for the L1 method.

Thus, the remainder of the proof is devoted to the comparison of $r_B^m$ and the corresponding truncation error for the L1 method.
Note  that the truncation error of any interpolation-based method can be represented in terms of the corresponding interpolation error
(as, for example, shown below in \eqref{trun_int}) as follows:
\beq\label{trun_int_B}
\Gamma(1-\alpha)\,|r_B^m|
\le \alpha
\int_{0}^{t_m}
\!\!\!(t_m-s)^{-\alpha-1}|(B-\Pi B)(s)|\,\dd s.
\eeq
 Considering that our method is based on the backward-quadratic interpolation operator $\Pi$ of \eqref{eq:defdelta},
while the L1 method is based on the standard linear interpolation operator, which we shall denote by $I$, we would like to show that
$|B-\Pi B|$ enjoys the same bound as $|B-I B|$.
The bound for the latter is given in \citep[Lemma 2.3]{NataliaMeng}.
First, on $ (0,t_n)$, where $B$ is linear, $B-IB=0$, as well as $B-\Pi B=0$.
Next, on $(t_{n+1},t_m)$ one has
$|(B-I B)(s)|\lesssim (t_j-s)\tau_j\,t_{j-1}^{\alpha-3}$ $\forall\,s\in(t_{j-1},t_j)$, $j>n+1$ (where we used $|B''|\lesssim t_{j-1}^{\alpha-3}$
on $(t_{j-1},t_j)$).
Additionally,
for $s\in(t_{j-1},t_j)$ one notes that
$(IB-\Pi B)(s)=\frac12(t_j-s)(s-t_{j-1})(\Pi B)''$, where $|(\Pi B)''|=|B''(\xi_j)|\lesssim t_{j-2}^{\alpha-3}$ for some $\xi_j\in(t_{j-2},t_j)$.
Hence, $|(I B-\Pi B)(s)|\lesssim (t_j-s)\tau_j\,t_{j-1}^{\alpha-3}$ (as $t_{j-2}\simeq t_{j-1}$), i.e. one gets the same bound
for $I B-\Pi B$ as for $B-I B$. So $|B-\Pi B|$ on $(t_{n+1},t_m)$ enjoys the same interpolation bound as the one for $|B-I B|$ used in \citep[Lemma 2.7]{NataliaSean}.

It remains to consider $I B-\Pi B$ on $(t_n,t_{n+1})$,
which is more delicate, since here
 $\Pi$ uses the interpolation points $\{ t_{n-1},t_{n},t_{n+1} \}$, while $B'$ is discontinuous at $t_n$.
A direct evaluation shows that now 
$$
|(\Pi B)''|=\frac{2}{\tau_n+\tau_{n+1}}\Bigl|\frac{B(t_{n+1})-B(t_{n})}{\tau_{n+1}}-\frac{B(t_{n})-B(t_{n-1})}{\tau_{n}}\Bigr| \lesssim
(\tau_n+\tau_{n+1})^{-1}\,t_n^{\alpha-2},
$$
where we used  $|B'(s)|\lesssim  t_{n}^{\alpha-2}$.
The resulting bound $|(I B-\Pi B)(s)|\lesssim(t_{n+1}-s)\,t_n^{\alpha-2}$ is worse than $|B-I B|\lesssim(t_{n+1}-s)\tau_{n+1}\,t_{n}^{\alpha-3}$ on this single interval $(t_n,t_{n+1})$,
so the corresponding ingredient in the truncation error integral
 \eqref{trun_int_B} will be estimated separately:
\begin{align*}
\hat r_B^m {}:={}&
\alpha
\int_{t_n}^{t_{n+1}}
\!\!\!(t_m-s)^{-\alpha-1}|(I B-\Pi B)(s)|\,\dd s
\lesssim
(t_{n+1}/t_m)^{\alpha+1}\,t_n^{\alpha-2}
\int_{t_n}^{t_{n+1}}
\!\!\!(t_{n+1}-s)^{-\alpha}\,\dd s,
\intertext{where we also used the observation that $(t_{n+1}-s)/(t_m-s)\le t_{n+1}/t_m$.
Combining the above with $\int_{t_n}^{t_{n+1}}(t_{n+1}-s)^{-\alpha}\,\dd s\lesssim \tau_{n+1}^{1-\alpha}$ and $t_{n+1}\simeq t_n$, and then with \eqref{caputo_B},
one gets}
\hat r_B^m\lesssim{}&
(\tau_{n+1}/t_{n+1})^{1-\alpha}\cdot\bigl[(t_{n}/t_m)^{\alpha}\,t_m^{-1}\bigr]
=(\tau_{n+1}/t_{n+1})^{1-\alpha}\cdot \bigl[c_\alpha^{-1} \pt^\alpha_t B(t_m)\bigr].
\end{align*}
Recall that $W(t)$ is a linear combination of the prototype functions $B(t\,;t_n)$, while our ultimate task is to show that the bulk truncation error
$r_W^m=(\delta_t^\alpha-\pt_t^\alpha) W(t_m)$ is dominated by
$\frac12\pt_t^\alpha W(t_m)$.
Hence, it suffices for the remaining ingredient $\hat r_B^m$ of the truncation error $r_B^m$ to be
$\le\frac14 \pt_t^\alpha B(t_m)$.
For the latter to be true, note that
$\tau_{n+1}/t_{n+1}\lesssim (\tau_{1}/t_{n+1})^{1/r}\le (\tau_{1}/t_p)^{1/r}$, in view of \eqref{eq:meshstability} and $n\ge p\gtrsim1$,
while $\tau_{1}/t_p\le p^{-1}$ (in view of \eqref{rho_cond}) can be made arbitrarily small by choosing $p$ sufficiently small.
This observation completes the proof.
\end{proof}

\begin{corollary}\label{cor_lem1}
Under the assumptions of Theorem~\ref{th-euniformgrid} with $2\le K\lesssim 1$,
suppose, additionally, that  $U^0=U^1=\cdots=U^K=0$. Then one gets \eqref{eq:stability}.
\end{corollary}
\begin{proof}
Note that \eqref{eq:meshstability} combined with $K\simeq 1$ implies that $\tau_K\simeq \tau_1$ and $t_{m+K-1}\simeq t_m$.
Hence, it suffices to prove a version of \eqref{eq:stability}  for $\{U^j\}_{j\ge K-1}$ on the grid $\{t_j\}_{j\ge K-1}$.
Importantly, the interpolation operator $\Pi$, used in the definition \eqref{eq:defdelta}
  of $\delta_{t}^{\alpha} $, is linear on $(t_{K-1},t_K)$ in the context of the restricted grid $\{t_j\}_{j\ge K-1}$, but is backward-quadratic in the context
  of the original grid $\{t_j\}_{j\ge 0}$. Nevertheless, $\Pi U^K=0$ on $(t_{K-1},t_K)$ for both grids, so $\delta^\alpha_t U^m$ remains unchanged $\forall\, m\ge K$ when we consider the restricted grid. With this observation,
 an application of Lemma~\ref{lem1} to $\{U^j\}_{j\ge K-1}$ on the grid $\{t_j\}_{j\ge K-1}$
immediately yields the desired result.
\end{proof}

\begin{lemma}\label{lem2}
Under the assumptions of Theorem~\ref{th-euniformgrid} on the grid, with $2\le K\lesssim 1$, suppose that $U^0=0$ and
$|\delta_t^\alpha U^m|\lesssim 1$ $\forall\,m\le K$,  while $U^m=0$ $\forall\,m\ge K+1$.
Then $|U^m|\lesssim \tau_1^\alpha$ $\forall\,m\le K$ and $|\delta_t^\alpha U^m|\lesssim (\tau_1/t_m)^{\alpha+1}$ $\forall\,m\ge K+1$.
\end{lemma}
\begin{proof}
 As $K\lesssim 1$, by \eqref{eq:meshstability}, one has $\tau_{j}\simeq \tau_{1}$ and $t_{j}\simeq \tau_{1}$ for $\forall\, j \leq K+2$.
 Additionally, $\delta_t^\alpha U^m=\sum_{j=1}^{m^*} U^j\,\pt_t^\alpha(\Pi\phi^j)(t_m)$,
 where $m^*:=\min\{m, K\}$, and  $\{\phi^{j}(t_{k})\}$ are the standard hat functions (see Fig.\,\ref{fig_basis}\,(right)).
 In this context,
  in view of $\tau_{j}\simeq \tau_{1}$, each $|\Pi\phi^j|$ is $\lesssim 1$ and has support on $(t_{j-1},t_{j+2})$, with
 $(\Pi\phi^m)'\gtrsim\tau_m^{-1}$ on $(t_{m-1},t_m)$.
 Now, a standard calculation (using integration by parts of type \eqref{trun_int}) shows that
 $|\pt_t^\alpha(\Pi\phi^j)(t_m)|\lesssim  \tau_1\, t_m^{-\alpha-1}$ $\forall\,j<m^*$, with a particular case
 $|\pt_t^\alpha(\Pi\phi^j)(t_m)|\lesssim  \tau_1^{-\alpha}$ $\forall\,j<m\le K$.
 Additionally, one gets $\pt_t^\alpha(\Pi\phi^m)(t_m)\gtrsim \tau_1^{-\alpha}$ $\forall\,m\le K$.
 Hence, $\forall\,m\le K$
 our hypothesis $|\delta_t^\alpha U^m|\lesssim 1$ yields the first desired bound
 $|U^m|\lesssim \tau_1^{\alpha}+\sum_{j=1}^{m-1}|U^j|\lesssim \tau_1^{\alpha}$.
Combining the latter with
  $|\delta_t^\alpha U^m|\lesssim \tau_1\, t_m^{-\alpha-1}\sum_{j=1}^{K} |U^j|$
  $\forall\,m\ge K+1$ yields the remaining desired bound on $|\delta_t^\alpha U^m|$.
 \end{proof}

\begin{proof}[Proof of Theorem~\ref{th-euniformgrid}]
 It suffices to prove \eqref{eq:stability} for $\gamma\le \alpha$ (as, otherwise, $(\tau_1/t_m)^{\gamma+1}\le (\tau_1/t_m)^{\alpha+1}$ while $\mathcal{U}^{m}(\tau_{1};\gamma)= \mathcal{U}^{m}(\tau_{1};\alpha)$).
 Set $\mathring U^j:=U^j$ $\forall\, j\le K$ and $\mathring U^j:=0$ otherwise.
 Then $\{\mathring U^j\}_{j=0}^M$ satisfies the hypotheses of Lemma~\ref{lem2}.
 An application of this lemma yields
 $|U^m|=|\mathring U^m|\lesssim \tau_1^\alpha\simeq \mathcal{U}^{m}(\tau_{1};\gamma)$ $\forall\,m\le K$.
 Additionally, this lemma implies $\forall\,m> K$ that
 $|\delta_t^\alpha (U^m-\mathring U^m)|\le |\delta_t^\alpha U^m|+|\delta_t^\alpha \mathring U^m|$, where  $|\delta_t^\alpha U^m|\lesssim (\tau_1/t_m)^{\gamma+1}$ and $|\delta_t^\alpha \mathring U^m|\lesssim (\tau_1/t_m)^{\alpha+1}$.
 Recalling that  $\gamma\le \alpha$, one concludes that $|\delta_t^\alpha (U^m-\mathring U^m)|\lesssim (\tau_1/t_m)^{\gamma+1}$ $\forall\,m> K$.
 Finally, an application of Corollary~\ref{cor_lem1} to the discrete function $\{U^m-\mathring U^m\}$ yields the desired assertion on $|U^m-\mathring U^m|=|U^m|$
 $\forall\,m> K$.
\end{proof}

For the error estimation in the time-fractional parabolic case (with spatial derivatives), we shall need the following, more subtle, version of
Theorem~\ref{th-euniformgrid}.

\begin{corollary}[{\citep{NataliaL2}}]\label{co-kW}
Let the temporal mesh satisfy \eqref{eq:meshstability}  for some $r\geq 1$, and also \eqref{rho_cond}
$\forall\,j\ge 1$ with a sufficiently small $\eps$.
Additionally,
let $\beta=1-2^{-\alpha}$, and
$\{\kappa_{m,j}\}$ be the unique set of the coefficients in the corresponding representation \eqref{eq:newrepresentation} for the operator $\delta_{t}^{\alpha}$
of \eqref{eq:defdelta}.
  Then,   given $\gamma\in\mathbb{R}$, the following is true for $\{U^{j}\}_{j=0}^{M}$
  and $\{W^{j}\}_{j=0}^{M}$ with $U^{0}=W^{0}=0$:
    \begin{align}\label{eq:mathcalW}
        \left.
         \begin{array}{cl}
            \sum_{j=0}^{m} \kappa_{m,j}\, W^{j}\lesssim (\tau_{1}/t_{m})^{\gamma+1}&\forall\, m\geq 1\\[4pt]
            \frac{|U^{m}|-\beta|U^{m-1}|}{1-\beta}\lesssim W^{m}\ &\forall \,\geq 1
         \end{array}
         \right\}\Rightarrow
        |U^{m}|\lesssim\mathcal{U}^{m}(\tau_{1};\gamma)\quad\forall\,m\ge 1,
\end{align}
where $\mathcal{U}^{m}$ is defined in \eqref{eq:stability}.
\end{corollary}

\begin{proof}
A version of this result is given in \cite[Theorem 3.2${}^*$]{NataliaL2}, while we sketch the proof here for completeness.
First, in view of \eqref{rho_cond}, Theorem~\ref{theo_beta_repr} implies that
the representation \eqref{eq:newrepresentation} of the discrete operator $\delta_t^\alpha$
of \eqref{eq:defdelta}
satisfies
 \eqref{eq:condition_kappa} with $\beta:=1-2^{-\alpha}$.
Hence,  $\delta_t^\alpha$ is associated with a product of two inverse-monotone matrices as follows.
Set $\delta_t^\alpha U^m_*=F^m:=(\tau_1/ t_m)^{\gamma+1}$ $\forall\,m\ge1$ subject to $U^0_*=0$.
Now  \eqref{eq:newrepresentation}\ yields
$\vec{F}=A_1\vec{V}_*$ with $\vec{V}_*=A_2\vec{U}_*$, or simply $\vec{F}=A_1 A_2\vec{U}_*$,
where
$A_1$ and $A_2$ are
 $M\times M$ matrices, and the notation of type
 $\vec{F}:=\{F^j\}_{j=1}^M$ is used for the corresponding column vectors.
  Being M-matrices in view of \eqref{eq:condition_kappa} (i.e. diagonally dominant, with non-positive off-diagonal elements),
both $A_1$ and $A_2$ are inverse-monotone, hence the product $A_1A_2$
is also inverse-monotone (i.e. the elements of its inverse are non-negative).
Next, note that the hypotheses in \eqref{eq:mathcalW} are equivalent to
$A_1\vec{W}\lesssim \vec{F}$ and $A_2\vec{|U|}\lesssim\vec{W}$. 
Since $A_1$ and $A_2$ are inverse-monotone,  $\vec{W}\lesssim A_1^{-1}\vec{F}$, and then $\vec{|U|}\le A_2^{-1}\vec{W}\lesssim A_2^{-1}A_1^{-1}\vec{F}=\vec{U}_*$,
or
$|U^m|\lesssim {U}^m_*$ $\forall\,m$.
As Theorem~\ref{theo_beta_repr}  implies that $0\le {U}^m_*\lesssim \mathcal{U}^{m}$, one gets the desired assertion.
\end{proof}

\section{Truncation error of the L2 fractional-derivative operator. Error estimation for
the case without spatial derivatives}\label{sec_trunc}

Following \citep{NataliaSean}, the exact solution will be assumed to have a singularity of type $t^\sigma$ for any fixed  $\sigma \in (0,1) \cup (1,2)$.
As a particular case, this includes $\sigma=\alpha$, which has attracted considerable attention in the literature.

\begin{lemma}[truncation error for $\delta_t^\alpha$]\label{lem_r}
Let $u=u(t)$ satisfy
$\textstyle\bigl|\frac{\dd^l}{\dd t^l}u(t)\bigr|\lesssim 1+t^{\sigma-l}$ for $ l=0,1,2,3$ and $t\in(0,T]$,
 with some $\sigma \in (0,1) \cup (1,2)$.
Then for the truncation error $r^m:=\delta_{t}^{\alpha} u(t_m)-\pt_t^\alpha u(t_m)$
on the temporal grid~\eqref{eq:meshstability}, one has
\beq\label{r_bound_simple}
|r^m|\lesssim  \tau^{\sigma - \alpha} (\tau/t_m)^{\gamma+1}\quad\forall\,m\ge 1,
\quad\mbox{where}\;\;
\textstyle
\gamma:=\min\bigl\{\alpha,\,\,\frac{3-\alpha}{r}+\alpha-\sigma-1\bigr\}.
\eeq
If $\sigma=\alpha$, then \eqref{r_bound_simple} simplifies to
\beq\label{r_bound_simple_alpha}
|r^m|\lesssim   (\tau/t_m)^{\gamma+1}\quad\forall\,m\ge 1,
\quad\mbox{where}\;\;
\textstyle
\gamma:=\min\bigl\{\alpha,\,\,\frac{3-\alpha}{r}-1\bigr\}.
\eeq
\end{lemma}

\begin{proof}
The proof is along the lines of \cite[Lemma~3.1]{NataliaSean}, where the truncation error of the L1 method was estimated, while now we extend the latter proof to the case of the higher-order L2 method.
First, $r^m$ is represented as $r^m=\pt_t^\alpha(-\chi)(t_m)$ via the function $\chi := u - \Pi u$, where $\Pi u(t)$ is the interpolant of $u(t)$ used in the definition \eqref{eq:defdelta} of $\delta_t^\alpha$.
Combining $(\Pi u-u)'(s)=-\chi'(s)$ with \eqref{eq:def-Caputoderivative} and \eqref{eq:defdelta}
and then integrating by parts yields
\beq\label{trun_int}
\Gamma(1-\alpha)\,r^m
=\int_0^{t_m}\!\!\!(t_m-s)^{-\alpha}\,\{-\chi'(s)\}\,\dd s
=\alpha
\int_0^{t_m}
\!\!\!(t_m-s)^{-\alpha-1}\chi(s)\,\dd s.
\eeq
Next,
for $s\in(t_{j-1},t_j)$ with $3\le j\le m$ we recall a standard quadratic interpolation error bound
$|\chi(s)|\lesssim (t_j-s)(s-t_{j-1})(s-t_{j-2})\,t_{j-2}^{\sigma-3}$. Combining the latter with
$t_{j-2}\simeq s \simeq t_j$ (which follows from \eqref{eq:meshstability}, as elaborated in \cite[Remark~2.1]{NataliaSean}) and
$\tau_j\lesssim\tau^{1/r}s^{1-1/r}$,
 (also, in view of \eqref{eq:meshstability}), one gets
\beq\label{chi1}
|\chi(s)|
\lesssim
\min \bigl\{\tau_j^3,\,(t_m-s)\,\tau_j^2\bigr\}\, s^{\sigma-3}
\lesssim \min\bigl\{
\tau^{3/r}s^{\sigma-3/r},\,
(t_m-s)\, \tau^{2/r}s^{\sigma-1-2/r}
\bigr\}.
\eeq
Additionally, we claim that
\beq\label{chi0}
|\chi(s)|\lesssim (t_j-s)\,(s^{\sigma-1}+\tau^{\sigma-1})\quad\mbox{for~}s\in(t_{j-1},t_j),\;\;j=1,2,
\eeq
(for $j=1$, i.e. on $(0,t_1)$, where $\Pi u$ is a linear interpolant of $u$, this bound can be found in \cite[(3.6)]{NataliaSean}).
To prove \eqref{chi0} for $\sigma\in(0,1)$, note that $\chi(t_j)=0$ implies $|\chi(s)|\le (t_j-s)\,\sup_{(s,t_j)}|\chi'|$, where
$\chi'=u'+(\Pi u)'$ with $|u'(s)|\lesssim s^{\sigma-1}$ and $|(\Pi u)'(s)|\lesssim \tau^{\sigma-1}$ (in view of $t_2\simeq t_1=\tau$).
For $\sigma\in(1,2)$, note that $\chi=\tilde u-\Pi \tilde u$, where $\tilde u(t):=u(t)-tu'(0)$. Now
a similar argument yields \eqref{chi0}
since  $\tilde u'(0)=0$, so $|\tilde u'(t)|\le \int_0^t|\tilde u''(s)|\,\dd s\lesssim t^{\sigma-1}$
(in contrast to $|u'(t)|\lesssim 1+t^{\sigma-1}$), so we again get \eqref{chi0}.
Finally, note that for $j=2$, \eqref{chi0} implies \eqref{chi1} (as $s\in(t_1,t_2)$ implies $s\simeq\tau$), so we shall use \eqref{chi1} for $s>\tau$
and
\eqref{chi0} for $s<\tau$, exactly as in the proof of \cite[Lemma~3.1]{NataliaSean}.
Imitating the latter proof, split $\int_0^{t_m}$ in \eqref{trun_int} into a sum of integrals on
$(0,t_1)$, $(t_1,t_m^*)$, and $(t_m^*,t_m)$, where  $t_m^*\simeq t_m-\tau_m$ for $m\ge 2$
(see \cite{NataliaSean} for further details), which leads to the bound
\begin{align*}
|r^m|&{}\lesssim
\tau^{\sigma-\alpha}(\tau/t_m)^{\alpha +1}+
 t_m^{\sigma-\alpha}\,
\bigr[(\tau/t_m)^{\sigma+1}
+
(\tau/t_m)^{(3-\alpha)/r}\bigr]
\\&{}\lesssim
\tau^{\sigma-\alpha}
\bigl[
(\tau/t_m)^{\alpha +1}+
(\tau/t_m)^{(3-\alpha)/r+\alpha-\sigma}
\bigr].
\end{align*}
As here $(\tau/t_m)\le 1$,  the desired assertion \eqref{r_bound_simple} and its particular case \eqref{r_bound_simple_alpha} follow.
\end{proof}

Now we are prepared to estimate
the error for
a simplest fractional-derivative problem without spatial derivatives. Consider
   \begin{subequations}\label{eq:simplest}
        \begin{align}
        \pt_{t}^{\alpha} u(t) = f(t)\qquad &\hbox{for}\ t\in(0,T],  \qquad\ \ u(0)=u_{0},\label{eq:simplestode}\\
        \delta_{t}^{\alpha}U^{m} = f(t_{m})\qquad &\hbox{for}\ m=1,\dots,M,\quad U^{0} =u_{0}.
        \label{eq:simplestdisode}
        \end{align}
    \end{subequations}

\begin{theorem}[pointwise-in-time error bound]\label{the_er_ode}
Let the solution $u$ of \eqref{eq:simplestode} satisfy
$\textstyle\bigl|\frac{\dd^l}{\dd t^l}u(t)\bigr|\lesssim 1+t^{\sigma-l}$ for $ l=0,1,2,3$ and $t\in(0,T]$,
 with some $\sigma \in (0,1) \cup (1,2)$.
Additionally, let the temporal mesh in \eqref{eq:simplestdisode} satisfy \eqref{eq:meshstability}  for some $r\geq 1$, and also \eqref{rho_cond}
with a sufficiently small $\eps$
$\forall\,j\ge K$ for some fixed $1\le K\lesssim 1$ (i.e. $K$ may be sufficiently large independently of $M$).
Then there exists a unique solution of \eqref{eq:simplestdisode} and $\forall\,m\ge 1$ one has
\begin{align}
\label{E_cal_m}
|u(t_m)-U^m|&\lesssim {\mathcal E}^m\!:=\!
\left\{\!\!\begin{array}{ll}
M^{-\nu r}\,t_m^{\alpha-1}&\mbox{if~}1\le r<\frac{3-\alpha}{\nu},\\[0.2cm]
 M^{-(3-\alpha)}\,t_m^{\alpha-1}\,\ell_0
&\mbox{if~}r=\frac{3-\alpha}{\nu}\ge1,\\[0.1cm]
M^{-(3-\alpha)}\,t_m^{\sigma-\frac{3-\alpha}{r}}
&\mbox{if~}r>\frac{3-\alpha}{\nu},
\end{array}\right.    
\intertext{where $\nu:=1+\sigma-\alpha$, $\ell_0=\ell_0(t_m)=[1+\ln(t_m/t_1)]$.
If $\sigma=\alpha$, then \eqref{E_cal_m} simplifies to}
\label{E_cal_m_alpha}
|u(t_m)-U^m|&\lesssim {\mathcal E}_\alpha^m\!:=\!
\left\{\!\!\begin{array}{ll}
M^{- r}\,t_m^{\alpha-1}&\mbox{if~}1\le r<3-\alpha,\\[0.2cm]
 M^{-(3-\alpha)}\,t_m^{\alpha-1}\,\ell_0
&\mbox{if~}r=3-\alpha,\\[0.1cm]
M^{-(3-\alpha)}\,t_m^{\alpha-\frac{3-\alpha}{r}}
&\mbox{if~}r>3-\alpha.
\end{array}\right.    
\end{align}
\end{theorem}
\begin{proof}
From \eqref{eq:simplest}, for $e^m:=U^m-u(t_m)$, one immediately gets
$\delta_t^\alpha e^m=r^m$, where for the truncation error
$r^m=\delta_{t}^{\alpha} u(t_m)-\pt_t^\alpha u(t_m)$, Lemma~\ref{lem_r}
yields the bound \eqref{r_bound_simple}.
Now, the bound~\eqref{eq:stability} in Theorem~\ref{th-euniformgrid}
gives
$|e^m|\lesssim
\widetilde\U^m
:=\tau^{\sigma - \alpha}\,\U^m$, i.e.
$\widetilde\U^m:=
\ell_\gamma\tau^\nu t_m^{\alpha-1}(\tau/t_m)^{\min\{0,\,\gamma\}}$, and it suffices to show that $\widetilde\U^m\simeq {\mathcal E}^m$.
As $\alpha>0$, one has $\min\{0,\,\gamma\}=\min\{0,\,\frac{3-\alpha}{r}-\nu\}$, and
$\ell_\gamma\neq 1$ iff $r\neq\frac{3-\alpha}{\nu} $ (which is equivalent to $\gamma\neq 0$).
Now, $r<\frac{3-\alpha}{\nu}$
yields
$\min\{0,\,\gamma\}=0$, so $\widetilde\U^m\simeq \tau^\nu t_m^{\alpha-1}\simeq M^{-\nu r}t_m^{\alpha-1}\simeq {\mathcal E}^m$.
Similarly,  $r=\frac{3-\alpha}{\nu}$ yields $\widetilde\U^m
\simeq \ell_0 M^{-\nu r}t_m^{\alpha-1}\simeq {\mathcal E}^m$.
Finally,  $r>\frac{3-\alpha}{\nu}$ yields
 $\min\{0,\,\gamma\}=\frac{3-\alpha}{r}-\nu<0$, so
 $\widetilde\U^m\simeq \tau^\nu t_m^{\alpha-1}(\tau/t_m)^{\frac{3-\alpha}{r}-\nu}\simeq \tau^{\frac{3-\alpha}{r}}t_m^{\sigma-\frac{3-\alpha}{r}}$
so again $\widetilde\U^m\simeq M^{-(3-\alpha)} t_m^{\sigma-\frac{3-\alpha}{r}}\simeq {\mathcal E}^m$.
If $\sigma=\alpha$, clearly ${\mathcal E}^m$ simplifies to ${\mathcal E}^m_\alpha$.
\end{proof}

\begin{corollary}[convergence in positive time]\label{cor_t1_con}
For $t_m\gtrsim 1$ the error bound in \eqref{E_cal_m} involves
${\mathcal E}^m\simeq M^{-\min\{\nu r,\,3-\alpha\}}$ if $r\neq\frac{3-\alpha}{\nu}$
and ${\mathcal E}^m\simeq M^{-(3-\alpha)}\ln M$ otherwise.
Hence, $r>\frac{3-\alpha}{\nu}$ yields the optimal convergence rate with  ${\mathcal E}^m\simeq M^{-(3-\alpha)}$ in positive time.
\end{corollary}

\begin{corollary}[global convergence]\label{cor_glob_con}
The global-in-time version of \eqref{E_cal_m} is
$\max_{t_m\le T}|u(t_m)-U^m|\lesssim
M^{-\min\{\sigma r,\,3-\alpha\}}$.
Hence, $r\ge \max\bigl\{\frac{3-\alpha}{\sigma},\,1\bigr\}$ yields the optimal global convergence rate $M^{-(3-\alpha)}$.
\end{corollary}

\begin{proof}
If $\sigma r\ge 3-\alpha$, one gets the case $r>\frac{3-\alpha}{\nu}$ in \eqref{E_cal_m},
so ${\mathcal E}^m\lesssim M^{-(3-\alpha)}$ $\forall\, m\ge 1$.
Otherwise,  ${\mathcal E}^m$ involves a negative power of $t_m$, so $\max_{m\le M}{\mathcal E}^m\lesssim{\mathcal E}^1$,
and a calculation (using $t_1\simeq M^{-r}$) shows that ${\mathcal E}^1\lesssim M^{-\sigma r}$ in each of the three cases in \eqref{E_cal_m}.
\end{proof}

\section{Error estimation for linear subdiffusion equations}\label{sec_err_linear}

\subsection{Error bounds for semi-discretizations in time}
Now we are prepared to address the semi-discretizations of our problem \eqref{eq:problem} in time using the discrete fractional-derivative operator $\delta_{t}^{\alpha}$ from \eqref{eq:defdelta}:
\begin{align}\label{eq:defsemidis}
  \delta_{t}^{\alpha} U^{m} + \mathcal{L} U^{m} = f(\cdot,t_{m})\  \hbox{in}\  \Omega,\quad U^{m}=0 \ \hbox{on} \ \pt\Omega \quad \forall \ m=1,\dots,M,\quad U^{0}=u_{0}.
\end{align}
We start with a version of Theorem~\ref{th-euniformgrid} for parabolic case.

\begin{lemma}[stability for parabolic case] \label{le:stable-semi}
Let the temporal mesh satisfy \eqref{eq:meshstability}  for some $r\geq 1$, and also \eqref{rho_cond}
with a sufficiently small $\eps$
$\forall\,j\ge K$ for some fixed $1\le K\lesssim 1$ (i.e. $K$ may be sufficiently large independently of $M$).
  Then,   given $\gamma\in\mathbb{R}$, the following is true for $\{U^{j}\}_{j=0}^{M}$ from \eqref{eq:defsemidis}:
     \begin{align}\label{eq:semidis_stability}
        \left.
        \begin{array}{c}
           \|f(\cdot,t_{m})\|_{L_2(\Omega)} \lesssim (\tau_{1}/t_{m})^{\gamma+1} \;\\[4pt]
            \;\forall m\geq 1,\quad U^{0}=0
        \end{array}\!\!\!
        \right\}
         \;\;\Rightarrow\;\; \|U^{m}\|_{L_2(\Omega)}  \lesssim \mathcal{U}^{m}(\tau_{1};\gamma)=\ell_{\gamma}\, \tau_{1}\,t_{m}^{\alpha-1} (\tau_{1}/t_{m})^{\min\{0,\gamma\}},
    \end{align}
where $\ell_{\gamma}=\ell_{\gamma}(t_{m}):=1+\ln(t_{m}/\tau_{1})$ for $\gamma=0$, and $\ell_{\gamma}:=1$ otherwise.
\end{lemma}

\begin{proof}
(i)
First, let us establish the desired result for the case $K=1$.
Then, in view of Theorem~\ref{theo_beta_repr}, the representation \eqref{eq:newrepresentation} of the discrete operator $\delta_t^\alpha$ of \eqref{eq:defdelta} satisfies \eqref{eq:condition_kappa} with $\beta:=1-2^{-\alpha}$, and furthermore, we can employ Corollary~\ref{co-kW}.
  Set
  $V^0:=U^0=0$, and also
  $V^{m}:= \frac{U^{m}-\beta U^{m-1}}{1-\beta}$ and $f^{m}:=f(\cdot,t_{m})$ for $m\ge1$, and then rewrite \eqref{eq:defsemidis} using
  \eqref{eq:newrepresentation},\,\eqref{eq:condition_kappa}
   as
  \begin{align*}
    \kappa_{m,m} V^{m} + \mathcal{L} U^{m} = |\kappa_{m,m-1}| V^{m-1} +\sum_{j=1}^{m-2} |\kappa_{m,j}| V^{j} +f^{m}.
  \end{align*}
  Consider the $L_2(\Omega)$ inner product, denoted by $\langle\cdot,\cdot\rangle$, of the above with $V^{m}$,  Then, using the notation
  \[
        W^{m}:=\sqrt{\| V^{m} \|_{L_{2}(\Omega)}^{2}+\frac{\kappa_{m,m}^{-1}}{1-\beta} \langle \mathcal{L}U^{m},U^{m} \rangle},
  \]
  one gets
  \begin{align}\label{Qm_eq}
    \kappa_{m,m}(W^{m})^{2} -\frac{\beta}{1-\beta} \langle \mathcal{L}U^{m},U^{m-1} \rangle = & 
    |\kappa_{m,m-1}|\, \langle V^{m},V^{m-1} \rangle
    +\sum_{j=1}^{m-2} |\kappa_{m,j}|\, \underbrace{\langle V^{m},V^{j} \rangle}_{\leq W^{m}W^{j}} + \underbrace{\langle V^{m},f^{m} \rangle}_{\leq W^{m}\|f^m\|_{L_{2}(\Omega)}}.
  \end{align}
Here we used $\langle V^m,V^j \rangle\le \| V^m \|_{L_{2}(\Omega)}\| V^{j} \|_{L_{2}(\Omega)}\le W^m\,W^j$, and, similarly,
  $\langle V^{m},f^{m} \rangle\le W^{m}\|f^m\|_{L_{2}(\Omega)}$.
  Next, we claim that (assuming that $\eps$) in \eqref{rho_cond}
is sufficiently small) one has
\begin{align}\label{Qm}
|\langle V^{m},V^{m-1} \rangle|+|\kappa_{m,m-1}|^{-1}\frac{\beta}{1-\beta} |\langle \mathcal{L}U^{m},U^{m-1} \rangle|\le W^{m}\,W^{m-1}.
  \end{align}
  For $m=1$, this is straightforward as
since $V^0=U^0=0$. For $m\ge 2$, to get~\eqref{Qm}, it suffices to check
  $|\kappa_{m,m-1}|^{-1}\frac{\beta}{1-\beta} \le \bigl\{\frac{\kappa_{m,m}^{-1}}{1-\beta}\, \frac{\kappa_{m-1,m-1}^{-1}}{1-\beta}\bigr\}^{1/2}$, or,
  equivalently,
 $\beta^2\,\kappa_{m-1,m-1}\, \kappa_{m,m}\le |\kappa_{m,m-1}|^2$.
  The latter is true assuming $\eps$ is sufficiently small in~\eqref{rho_cond}, in view of
  Lemma~\ref{lem_cond_aaa} combined with Remark~\ref{rem_kappa}.
Now that \eqref{Qm} is proved, combining it with \eqref{Qm_eq} yields
  $\kappa_{m,m}(W^{m})^{2}\le \sum_{j=1}^{m-1} |\kappa_{m,j}| W^m\,W^j +W^m\,\|f^{m}\|_{L_{2}(\Omega)}$, or, equivalently, (in view of $\kappa_{m,j}\le 0$ for $j<m$)
\begin{equation}\label{W_U}
  \sum_{j=1}^{m} \kappa_{m,j} W^j\le \|f^{m}\|_{L_{2}(\Omega)},
  \quad\mbox{where~~}
  W^m\ge \| V^{m} \|_{L_{2}(\Omega)}\ge \frac{\|U^{m}\|_{L_{2}(\Omega)}-\beta \|U^{m-1}\|_{L_{2}(\Omega)}}{1-\beta}.
\end{equation}
  Finally, an application of Corollary~\ref{co-kW} yields the desired assertion \eqref{eq:semidis_stability}.

  (ii)
  To extend the desired result to the case $2\le K\lesssim 1$, we proceed exactly as in the proof of Theorem~\ref{th-euniformgrid}, i.e.
  establish versions of Corollary~\ref{cor_lem1} (in which is it assumed that $U^0=U^1=\cdots=U^K=0$)
  and Lemma~\ref{lem2} (in which it is assumed that $U^m=0$ $\forall\,m\ge K+1$, and one now shows that
 $\|U^m\|_{L_{2}(\Omega)}+  \sqrt{\tau_1^{\alpha}\langle \mathcal{L}U^{m},U^{m} \rangle } \lesssim \tau_1^{\alpha}+\sum_{j=1}^{m-1}\|U^j\|_{L_{2}(\Omega)}\lesssim \tau_1^{\alpha}$  $\forall\,m\le K$) for the parabolic case.
\end{proof}

\begin{theorem}[pointwise-in-time error bound]\label{the_er_semi}
Let the solution $u$ of \eqref{eq:problem} satisfy
$\textstyle\|\pt_t^l u(\cdot,t)\|_{L_2(\Omega)}\lesssim 1+t^{\sigma-l}$ for $ l=0,1,2,3$ and $t\in(0,T]$,
 with some $\sigma \in (0,1) \cup (1,2)$.
Additionally, let the temporal mesh in \eqref{eq:defsemidis} satisfy \eqref{eq:meshstability}  for some $r\geq 1$, and also \eqref{rho_cond}
with a sufficiently small $\eps$
$\forall\,j\ge K$ for some fixed $1\le K\lesssim 1$ (i.e. $K$ may be sufficiently large independently of $M$).
Then there exists a unique solution of \eqref{eq:defsemidis} and $\forall\,m\ge 1$ one has
\begin{equation}
\label{E_cal_m_par}
\|u(t_m)-U^m\|_{L_2(\Omega)}\lesssim {\mathcal E}^m\qquad \forall\,m\ge 1,
\end{equation}
where ${\mathcal E}^m$ is defined in \eqref{E_cal_m}.
If $\sigma=\alpha$, then ${\mathcal E}^m$ simplifies to ${\mathcal E}_\alpha^m$ defined in \eqref{E_cal_m_alpha}.
\end{theorem}

\begin{proof}
It suffices to note that the error $e^m:=U^m-u(t_m)$ satisfies a version of \eqref{eq:defsemidis}
with $f(\cdot,t_{m})$
replaced by the truncation error $r^m=\delta_{t}^{\alpha} u(\cdot,t_m)-\pt_t^\alpha u(\cdot,t_m)$.
Since for the latter one gets
$\|r^m\|_{L_2(\Omega)}\lesssim  \tau_{1}^{\sigma - \alpha} (\tau_{1}/t_m)^{\gamma+1}$
(similarly to Lemma~\ref{lem_r}), an application of Lemma~\ref{le:stable-semi}
yields the desired error bound
(which can be shown along the lines of the proof of Theorem~\ref{the_er_ode}).
\end{proof}

\subsection{Error bounds for full discretizations}
Let $S_{h}\subset H_{0}^{1}\cap C(\bar{\Omega})$ be a Lagrange finite element space of fixed degree $s\geq 1$,
 relative to a
 quasiuniform simplicial triangulation
 $\mathcal T$ of the spatial domain $\Omega$.
Now for all $m=1,\dots,M$, let $u_{h}^{m}\in S_{h}$ satisfy
\begin{align}\label{eq:fullydis}
  \langle \delta_{t}^{\alpha} u_{h}^{m},v_{h} \rangle + A( u_{h}^{m},v_{h} )= \langle f(\cdot,t_{m}),v_{h} \rangle\quad \forall v_{h}\in S_{h}
\end{align}
Here, $\langle \cdot,\cdot \rangle$ is the $L_{2}(\Omega)$ inner product and $A$ is the standard symmetric bilinear form $A(v,w)=\langle \mathcal{L}v,w \rangle$.

To estimate the error of the above full discretization,
 define a Ritz projection $\mathcal{R}_{h}u(t)\in S_{h}$ of $u$ by $A(\mathcal{R}_{h}u,v_{h})=\langle \mathcal{L}u,v_{h} \rangle$ $\forall\, v_{h}\in S_{h}$.
 For the error of the Ritz projection $\rho(\cdot,t):=\mathcal{R}_{h}u(t)-u(\cdot,t)$ one typically gets
 \begin{align*}
   \| \rho(\cdot,t) \|_{L_{2}(\Omega)}&\lesssim h \inf_{v_{h}\in S_{h}} \|  u(\cdot,t) -v_{h} \|_{H^{1}(\Omega)},\\
     \|\delta_t^\alpha\rho(\cdot,t_m) \|_{L_{2}(\Omega)}&\lesssim h \inf_{v_{h}\in S_{h}} \|  \delta_t^\alpha u(\cdot,t_m) -v_{h} \|_{H^{1}(\Omega)},
 \end{align*}
where the first bound can be found in  \cite[Theorem 5.7.6]{BrennerScott},
while the second bound follows, in view of $\delta_t^\alpha\mathcal{R}_{h}=\mathcal{R}_{h}\delta_t^\alpha$ (since $\mathcal{L}$ is independent of $t$).

\begin{theorem}\label{th-error-fully}
Under the conditions of Theorem~\ref{the_er_semi} on the solution $u$ of \eqref{eq:problem} and on the temporal mesh,
for the  $\{ u_{h}^{m} \}$ from \eqref{eq:fullydis} subject to $u_h^0=\mathcal{R}_h u_0$, one has
 \begin{align}
   \| u(\cdot,t_{m})- u_{h}^{m} \|_{L_{2}(\Omega)}\lesssim \mathcal{E}^{m} +  \| \rho(\cdot,t_{m}) \|_{L_{2}(\Omega)} +  \max_{m\ge 1} \Bigl\{ t_m^{\alpha}\,\| \delta_t^\alpha\rho(\cdot,t_m) \|_{L_{2}(\Omega)}\Bigr \}
 \end{align}
 for all $m\geq 1$, where $ \rho(\cdot,t):=\mathcal{R}_{h}u(\cdot,t)-u(\cdot,t)$, and $\mathcal{E}^{m}$ is defined in \eqref{E_cal_m}.
\end{theorem}
\begin{proof}
First, note that
for the fully discrete problem
one gets a version of the stability result \eqref{eq:semidis_stability} of Lemma~\ref{le:stable-semi}.
Now,
  let $e_{h}^{m}:=\mathcal{R}_{h}u(t_{m})-u_{h}^{m}\in S_{h}$. Then $u(\cdot,t_{m})-u_{h}^{m}=e_{h}^{m}-\rho(\cdot,t_{m})$,
   and a standard calculation using \eqref{eq:fullydis} and \eqref{eq:problem} yields
  \begin{align}
     \langle \delta_{t}^{\alpha} e_{h}^{m},v_{h} \rangle + A( e_{h}^{m},v_{h} )
     \nonumber=& \langle \delta_{t}^{\alpha} \mathcal{R}_{h}u(\cdot,t_{m}),v_{h} \rangle + A(\mathcal{R}_{h}u(\cdot,t_{m}),v_{h} )-\langle f(\cdot,t_{m}),v_{h} \rangle
     \\\label{eq:fully-error}
     =& \langle \delta_{t}^{\alpha} \rho(\cdot,t_{m})+r^{m},v_{h} \rangle\quad \forall v_{h}\in S_{h}.
  \end{align}
  Here $r^{m}=\delta_{t}^{\alpha} u(\cdot,t_{m})-\pt_{t}^{\alpha}u(\cdot,t_{m})$ is the truncation error, which induces $ \mathcal{E}^{m}$ in the error bound
  (which is shown exactly as in the proof of Theorem~\ref{the_er_semi}).
  To estimate the error induced by $\delta_{t}^{\alpha} \rho(\cdot,t_{m})$, note that
  $\|\delta_{t}^{\alpha} \rho(\cdot,t_{m})\|_{L_{2}(\Omega)}\le E\, t_m^{-\alpha}=E\, \tau_1^{-\alpha}\,(\tau_1/t_m)^{\alpha}$, where $E:=\max_{m\ge 1} \bigl\{ t_m^{\alpha}\,\| \delta_t^\alpha\rho(\cdot,t_m) \|_{L_{2}(\Omega)}\bigr \}$. Hence, using a version of the stability result \eqref{eq:semidis_stability}, this part of the error
  is $\lesssim E\, \tau_1^{-\alpha}\, \mathcal{U}^{m}(\tau_{1};\alpha-1)= E$.
\end{proof}

Note that using Theorem \ref{th-error-fully} and, additionally, assuming sufficient spatial regularity of $u$ and $\pt_t u$,
	namely
	\(\|u\|_{W_2^{s+1}(\Omega)}+
	\|\partial_t u\|_{W_2^{s+1}(\Omega)}\lesssim 1\),
	 one gets
  \begin{align}\label{co-convergence}
    \| u(\cdot,t_{m})- u_{h}^{m} \|_{L_{2}(\Omega)}&\lesssim  \mathcal{E}^{m} + h^{s+1}.
  \end{align}

\section{Extension to semilinear subdiffusion equations}\label{sec_semi}
Consider a semilinear version of \eqref{eq:problem} of the form
\begin{align}\label{eq:semiproblem}
     \pt_{t}^{\alpha} u + \mathcal{L} u = f(x,t,u)\ \hbox{for}\ (x,t)\in\Omega\times(0,T] ,\quad
      u(x,t) = 0 \ \hbox{for}\ (x,t)\in\partial\Omega\times(0,T],
\end{align}
 subject to $u(x,0)=u_{0}(x)$ $\forall\,x\in\Omega$,
and its semidiscretization 
in time
\begin{align}\label{eq:semidisforsemiprob}
  \delta_{t}^{\alpha} U^{m} + \mathcal{L} U^{m} = f(\cdot,t_{m},U^{m})\ \hbox{in}\ \Omega,\quad U^{m}=0 \ \hbox{on}\ \partial\Omega \quad \forall m = 1,\dots,M,
\end{align}
also subject to $U^{0}=u_{0}$, with the same L2 operator $\delta_{t}^{\alpha}$ from \eqref{eq:defdelta}.

Our purpose in this section is to extend Theorem~\ref{the_er_semi} to the semilinear case,
assuming that
$f$ is continuous and
$ \ |\partial_{u} f(\cdot,\cdot,\cdot)|\leq \lambda~ \hbox{in}~ \Omega\times(0,T]\times \mathbb{R}$,
for some positive constant $\lambda$. (Hence, without loss of generality, it may also be assumed that $c(x)=0$ in the definition \eqref{eq:def-ellipticspatial} of $\mathcal{L}$.
Note also that our assumption of $f$ may be weakened to a Lipschitz condition; see \cite[Assumption A1]{NataliaSemi}.)
Throughout this section, we make
a related additional condition on the temporal mesh that
\begin{equation}\label{tau_lambda}
  \lambda\tau_j^{\alpha}< \{\Gamma(2-\alpha)\}^{-1}\quad \forall\,j\ge1.
\end{equation}
Our first step is to show that if the temporal grid is such that one gets
a discrete comparison principle in the form of Corollary~\ref{th-discretecomparisonprinciple} for $\delta_t^\alpha$, one immediately enjoys
a similar discrete comparison principle for the operator $\delta_t^\alpha-\lambda$.

\begin{lemma}[discrete comparison principle for $\delta_t^\alpha-\lambda$]\label{semi-discretecomparisonprinciple}
Under the conditions of Corollary~\ref{th-discretecomparisonprinciple} on the temporal grid
and the additional condition \eqref{tau_lambda},
the following
statement is true:
\smallskip
%

If $U^{0}= 0$ and $(\delta_t^\alpha-\lambda)U^{m}\geq 0$ $\forall\, m\geq 1$, then $U^{j}\geq0$ for all $j\geq 0$.
%
\end{lemma}

\begin{proof}
We essentially follow the proof of \cite[Lemma~2.2]{NataliaSemi}.
As already discussed in the proof of Corollary~\ref{co-kW},
$(\delta_t^\alpha-\lambda)U^{m}\geq 0$ 
$\forall\, m\geq 1$ can be written in the matrix form as
$A_1 A_2\vec{U}\ge \lambda \vec{U}$,
where
$A_1$ and $A_2$ are
 $M\times M$ matrices, and the notation
 $\vec{U}:=\{U^j\}_{j=1}^M$ is used for the corresponding column vector.
Recall that, in view of \eqref{eq:condition_kappa}, both $A_1$ and $A_2$ are inverse-monotone, hence the product $A_1A_2$
is also inverse-monotone (i.e. the elements of its inverse are non-negative).
Hence, one gets
$\vec{U}\ge \lambda (A_1 A_2)^{-1}\vec{U}$, where $A_1$, $A_2$, and hence both
$A_1 A_2$ and $(A_1 A_2)^{-1}$, are lower triangular matrices. Denoting the elements of $(A_1 A_2)^{-1}$ by $a^*_{ij}$, one then gets
$(1-\lambda a^*_{mm})U^m\ge \lambda\sum_{j<m}a^*_{mj} U^j$, so the desired assertion is easily proved by induction
as long as $1-\lambda a^*_{mm}>0$ $\forall\,m\ge1$. The latter is ensured by \eqref{tau_lambda}
(as explained in the proof of \cite[Lemma~2.2]{NataliaSemi}; note also that the diagonal elements in the lower-triangular matrix $A_1 A_2$, and, hence, the diagonal elements in its inverse
$(A_1 A_2)^{-1}$, are exactly the same for our $\delta_t^\alpha$ and the forward L2 operator considered in \cite{NataliaSemi}).
\end{proof}

Now that we have established the above discrete comparison principle, 
we can extend our convergence analysis to the semilinear case, which ultimately yields the following error bound.

\begin{theorem}[error bound for semilinear case]\label{the_er_semilin}
Let the solution $u$ of \eqref{eq:semiproblem} satisfy
$\textstyle\|\pt_t^l u(\cdot,t)\|_{L_2(\Omega)}\lesssim 1+t^{\sigma-l}$ for $ l=0,1,2,3$ and $t\in(0,T]$,
 with some $\sigma \in (0,1) \cup (1,2)$.
Additionally, let the temporal mesh
in \eqref{eq:semidisforsemiprob} satisfy \eqref{tau_lambda} and the assumptions made in Theorem~\ref{the_er_semi}.
%
%
Then there exists a unique solution of \eqref{eq:semidisforsemiprob} and $\forall\,m\ge 1$ one has
\begin{equation}
\label{E_cal_m_semilin}
\|u(t_m)-U^m\|_{L_2(\Omega)}\lesssim {\mathcal E}^m\qquad \forall\,m\ge 1,
\end{equation}
where ${\mathcal E}^m$ is defined in \eqref{E_cal_m}.
If $\sigma=\alpha$, then ${\mathcal E}^m$ simplifies to ${\mathcal E}_\alpha^m$ defined in \eqref{E_cal_m_alpha}.
\end{theorem}

To prove the above theorem, we shall employ the following version of Lemma~\ref{le:stable-semi}.

\begin{lemma}[stability for semilinear parabolic case] \label{le:stable-semi_nonl}
Let the temporal mesh satisfy the conditions in Lemma~\ref{le:stable-semi}
and the additional condition \eqref{tau_lambda}.
%
%
  Then,   given $\gamma\in\mathbb{R}$, the following is true for $\{U^{j}\}_{j=0}^{M}$ from \eqref{eq:defsemidis}:
     \begin{align}\label{eq:semilin_stability}
        \left.
        \begin{array}{c}
           \|f(\cdot,t_{m}, U^m)\|_{L_2(\Omega)}-\lambda \|U^{m}\|_{L_2(\Omega)} \lesssim (\tau_{1}/t_{m})^{\gamma+1} \;\\[4pt]
            \;\forall m\geq 1,\quad U^{0}=0
        \end{array}\!\!\!
        \right\}
         \;\;\Rightarrow\;\;
         \!\!\!\begin{array}{l}\|U^{m}\|_{L_2(\Omega)}  \lesssim \mathcal{U}^{m}(\tau_{1};\gamma)\\
         \qquad{}=\ell_{\gamma}\, \tau_{1}\,t_{m}^{\alpha-1} (\tau_{1}/t_{m})^{\min\{0,\gamma\}},
         \end{array}\!\!\!
    \end{align}
where $\ell_{\gamma}=\ell_{\gamma}(t_{m})=1+\ln(t_{m}/\tau_{1})$ for $\gamma=0$, and $\ell_{\gamma}=1$ otherwise.
\end{lemma}

\begin{proof}
We essentially follow the proof of \cite[Lemma~3.1]{NataliaSemi}, so only sketch this proof.
First, we establish a more general version of Theorem~\ref{th-euniformgrid}
with $|\delta_t^\alpha U^m|$ replaced by $(\delta_t^\alpha-\lambda) U^m$, and $|U^m|$ replaced by $U^m$,
which holds true
under the additional condition \eqref{tau_lambda}.
This extension of the original version of Theorem~\ref{th-euniformgrid}
is proved along the lines of the proof of \cite[Theorem~2.2]{NataliaSean}. Importantly, the latter proof
relies on \cite[Lemma 3.2]{NataliaUpLow}; while the latter is presented for the L1 operator, the result also applies
  to our L2 operator $\delta_t^\alpha$
(with minimal changes in the proof).

Next, again under the additional condition \eqref{tau_lambda},
we get a more general version of  Corollary~\ref{co-kW}, in which the hypothesis
$\sum_{j=0}^{m} \kappa_{m,j}\, W^{j}\lesssim (\tau_{1}/t_{m})^{\gamma+1}$
is replaced by a weaker
$\sum_{j=0}^{m} \kappa_{m,j}\, W^{j}-\lambda |U^m|\lesssim (\tau_{1}/t_{m})^{\gamma+1}$.
The proof imitates the one of Corollary~\ref{co-kW}, only now
we set $F^m:=(\tau_1/ t_m)^{\gamma+1}$ $\forall\,m\ge1$ so 
our hypotheses are
$A_1\vec{W}-\lambda\vec{|U|}\lesssim \vec{F}$ and $A_2\vec{|U|}\lesssim\vec{W}$. 
The second hypothesis immediately yields
$\vec{|U|}\lesssim\vec{U}_*:= A_2^{-1}\vec{W}$. Combined with the first, this leads to
$(A_1A_2-\lambda) \vec{U}_*\lesssim \vec{F}$, or, equivalently, $(\delta_t^\alpha-\lambda) U_*^m\lesssim F^m=(\tau_1/ t_m)^{\gamma+1}$, while $|U^m|\le U_*^m$.  Finally, an application of the above more general version of Theorem~\ref{th-euniformgrid}
yields the desired, more general, version of Corollary~\ref{co-kW}.

To complete the proof of \eqref{eq:semilin_stability}, note that
one again gets \eqref{W_U}, only now with $f^{m}:=f(\cdot,t_{m}, U^m)$.
So it remains to subtract $\lambda \|U^{m}\|_{L_2(\Omega)} $
from the first relationship $\sum_{j=1}^{m} \kappa_{m,j} W^j\le \|f^{m}\|_{L_{2}(\Omega)}$ in \eqref{W_U},
and then apply the above more general version of Corollary~\ref{co-kW}.
\end{proof}

\begin{proof}[Proof of Theorem~\ref{the_er_semilin}]
A standard evaluation shows  that the error $e^m:=U^m-u(t_m)$ satisfies a version of \eqref{eq:semidisforsemiprob}
with $f(\cdot,t_{m}, U^m)$
replaced by $F(\cdot,t_{m}, e^m)+ r^m$, where
$F(\cdot,t_{m}, e^m):=f(\cdot,t_{m}, u(\cdot,t_m)+e^m)-f(\cdot,t_{m}, u(\cdot,t_m))$, so $|F(\cdot,t_{m}, e^m)|\le \lambda |e^m|$, while
 $r^m=\delta_{t}^{\alpha} u(\cdot,t_m)-\pt_t^\alpha u(\cdot,t_m)$ is the truncation error.
Since for the latter one gets
$\|r^m\|_{L_2(\Omega)}\lesssim  \tau^{\sigma - \alpha} (\tau/t_m)^{\gamma+1}$
(similarly to Lemma~\ref{lem_r}),
one concludes that
$\|F(\cdot,t_{m}, e^m)+ r^m\|_{L_2(\Omega)}-\lambda \|e^m\|_{L_2(\Omega)}\lesssim \tau^{\sigma - \alpha} (\tau/t_m)^{\gamma+1}$.
Hence an application of Lemma~\ref{le:stable-semi_nonl}
(instead of Lemma~\ref{le:stable-semi}, used in the linear case)
yields the desired error bound.
\end{proof}

	\section{Numerical experiments}\label{sec_numerics}
	
	\begin{example}[Initial-value test problem]\label{ex:simple}
		Consider the simplest initial-value fractional-derivative test problem \eqref{eq:simplest} with exact solution $u(t)=t^\sigma$ and $f(t)=\frac{\sigma+1}{\sigma-\alpha+1} t^{\sigma-\alpha}$.
	\end{example}
	
	In this example, we investigate the performance of our numerical scheme under different regularity conditions of the exact solution by varying the parameter $\sigma$. Specifically, we consider three representative cases: $\sigma = \alpha$, $\sigma = 2\alpha$, and $\sigma = \alpha/2$.
	
	First, we examine the pointwise error distributions. Figures \ref{fig-simple_sigmaEalpha}, \ref{fig-simple_sigmaE2alpha}, and \ref{fig-simple_sigmaEalphaD2} display the pointwise errors for the cases $\sigma = \alpha$ (tested with $\alpha=0.5$), $\sigma = 2\alpha$ (tested with $\alpha=0.3$), and $\sigma = \alpha/2$ (tested with $\alpha=0.7$), respectively. These specific combinations are selected to represent different singularity strengths.
	In each case, tests are conducted on meshes with $M=1024$ utilizing different grading parameters $r$. The numerical errors are compared with the pointwise theoretical error bounds defined in \eqref{E_cal_m} (represented by dashed lines).
the numerical pointwise errors agree remarkably well with the theoretical bounds across the entire time interval, except at a very few initial mesh nodes, where the pointwise error profile is somewhat different, while the error bound \eqref{E_cal_m} is still satisfied.
	
	Next, Table \ref{tb-simple-max} lists the maximum nodal errors and the corresponding computational convergence rates for the specific case $\sigma = \alpha$ with various values of $\alpha$. We test three different grading strategies: $r=1$, $r=3-\alpha$, and the optimally graded mesh with $r=\frac{3-\alpha}{\alpha}$. It can be  observed that the numerical convergence order is exactly $\min\{r\alpha,\ 3-\alpha\}$, which is in perfect agreement with our theoretical analysis.

\begin{figure}[!t]
	\centering
	\begin{minipage}{0.45\linewidth}
		\centering
		\includegraphics[width=\linewidth]{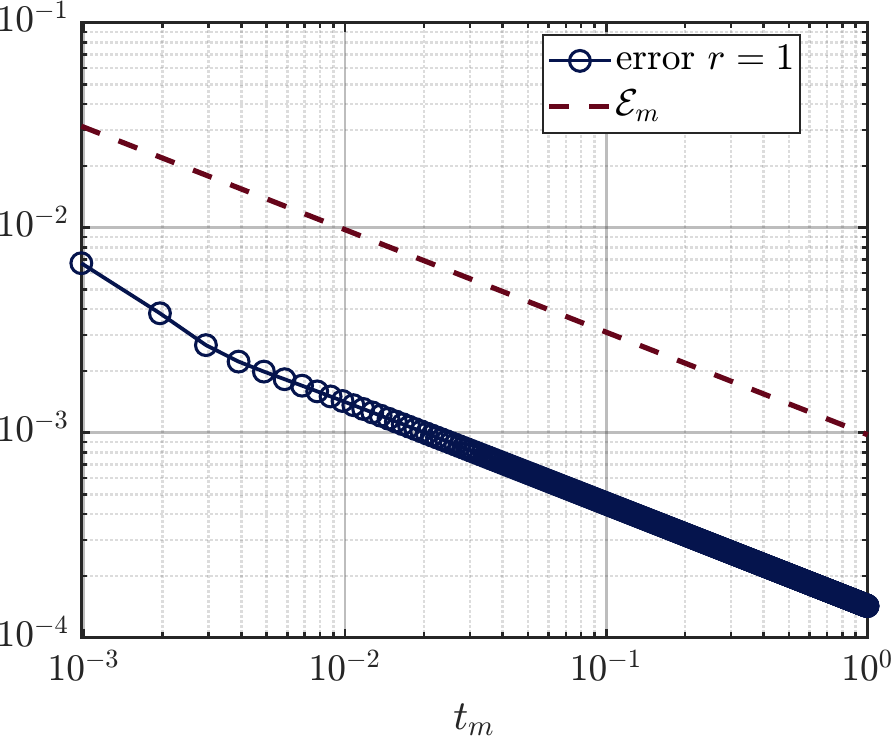}
	\end{minipage}%
	\hspace{0.8em}%
	\begin{minipage}{0.45\linewidth}
		\centering
		\includegraphics[width=\linewidth]{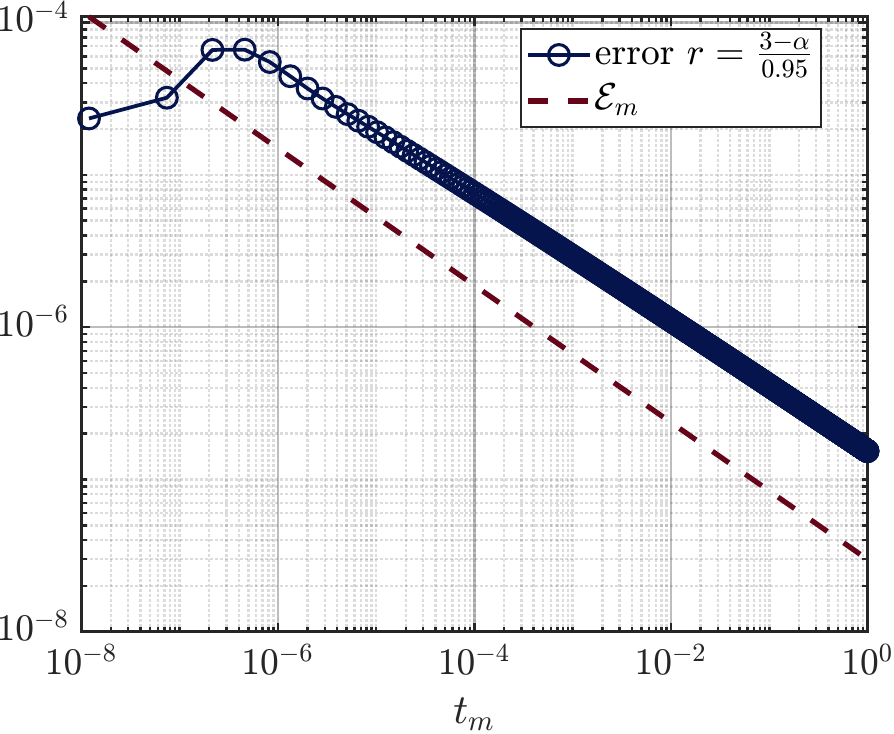}
	\end{minipage}
	
	\vspace{0.8em}
	
	\begin{minipage}{0.45\linewidth}
		\centering
		\includegraphics[width=\linewidth]{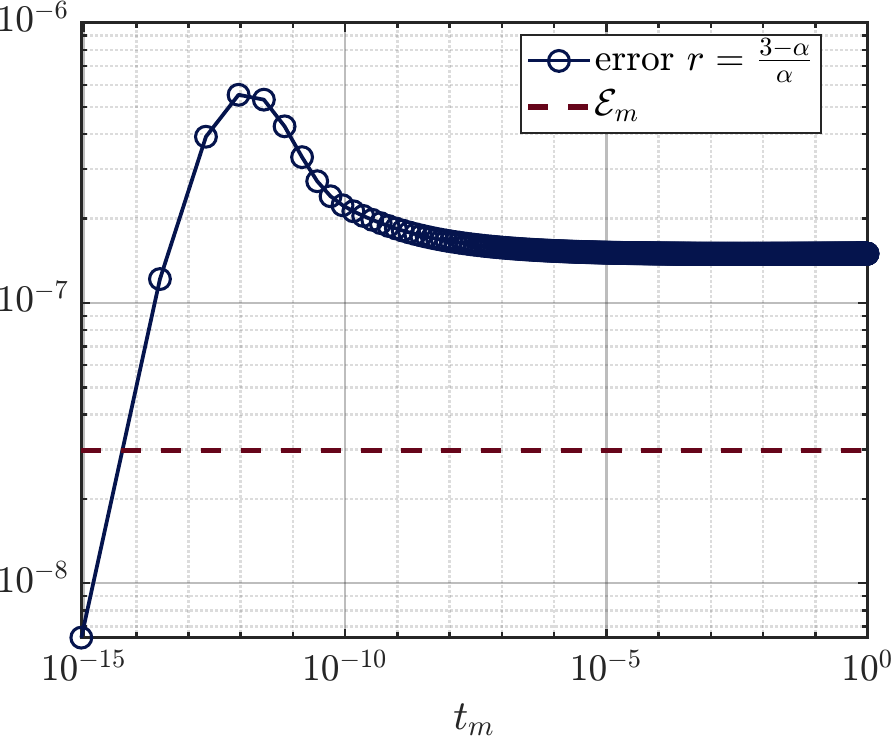}
	\end{minipage}%
	\hspace{0.8em}%
	\begin{minipage}{0.45\linewidth}
		\centering
		\includegraphics[width=\linewidth]{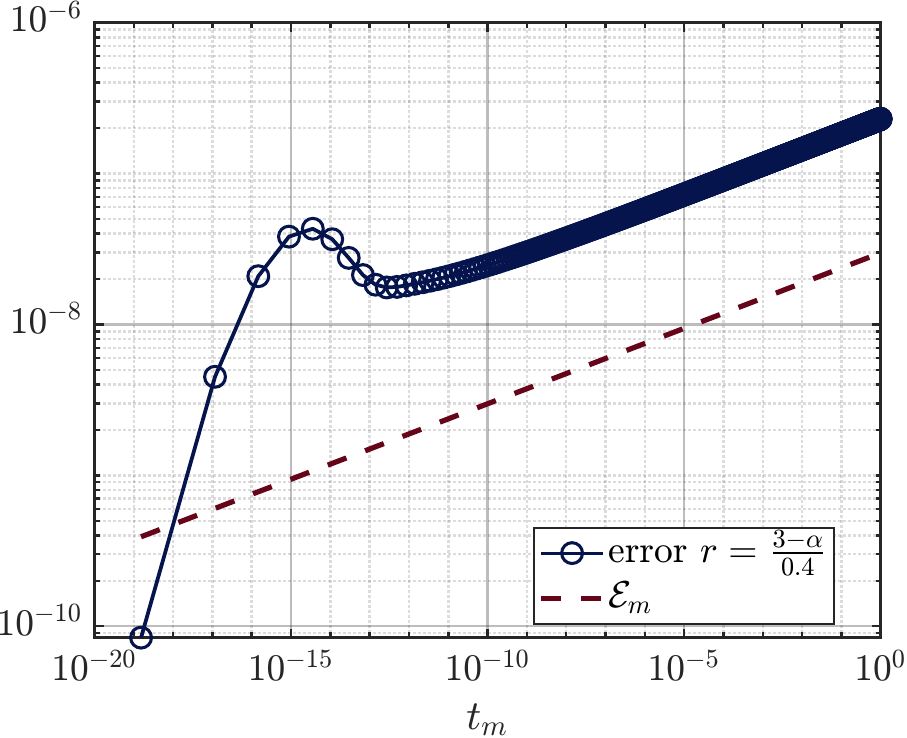}
	\end{minipage}
	
	\caption{Initial-value test problem in Example \ref{ex:simple} with $\sigma = \alpha$: Pointwise errors for $\alpha=0.5$ and $M=1024$, cases $r=1$, $r=\frac{3-\alpha}{0.95}$, $r=\frac{3-\alpha}{\alpha}$, $r=\frac{3-\alpha}{0.4}$.}
	\label{fig-simple_sigmaEalpha}
\end{figure}

\begin{figure}[!t]
	\centering
	\begin{minipage}{0.45\linewidth}
		\centering
		\includegraphics[width=\linewidth]{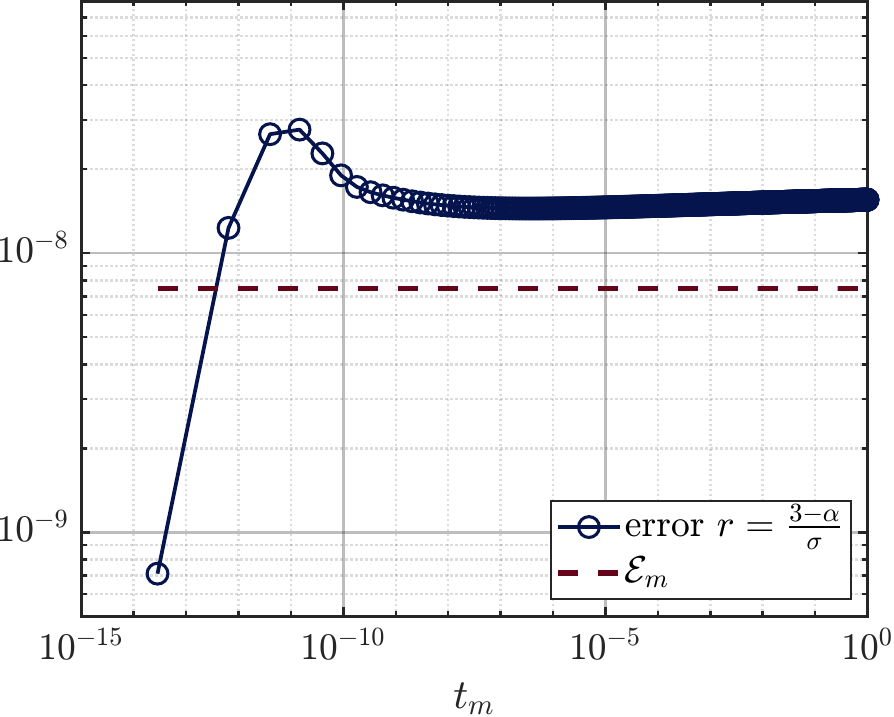}
	\end{minipage}%
	\hspace{0.8em}%
	\begin{minipage}{0.45\linewidth}
		\centering
		\includegraphics[width=\linewidth]{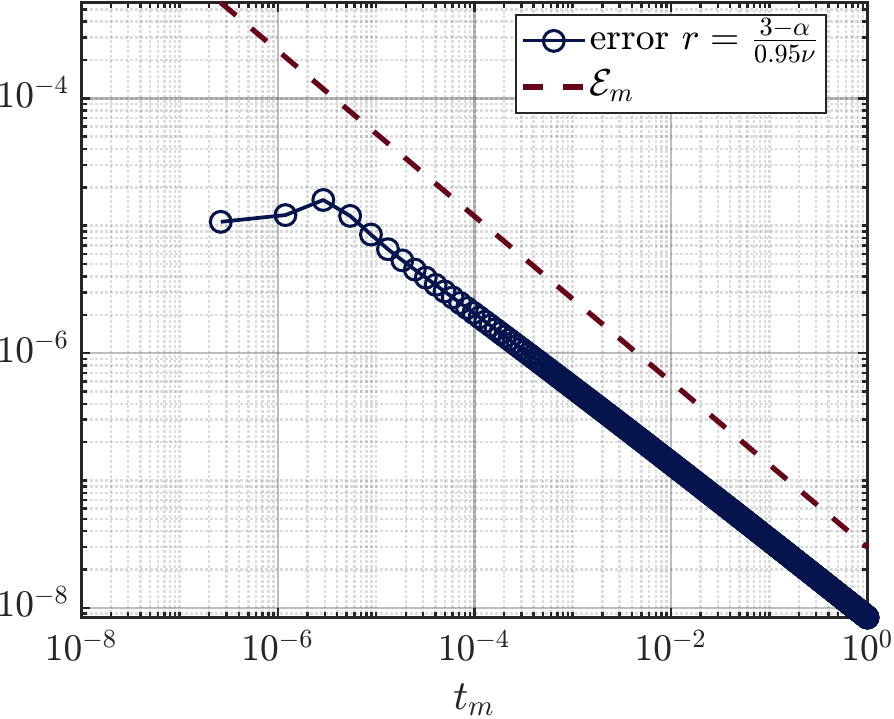}
	\end{minipage}
	\caption{Initial-value test problem in Example \ref{ex:simple} with $\sigma = 2\alpha$: Pointwise errors for $\alpha=0.3$ and $M=1024$, cases $r=\frac{3-\alpha}{\sigma}$ (left) and $r=\frac{3-\alpha}{0.95\nu}$ (right).}
	\label{fig-simple_sigmaE2alpha}
\end{figure}

\begin{figure}[!t]
	\centering
	\begin{minipage}{0.45\linewidth}
		\centering
		\includegraphics[width=\linewidth]{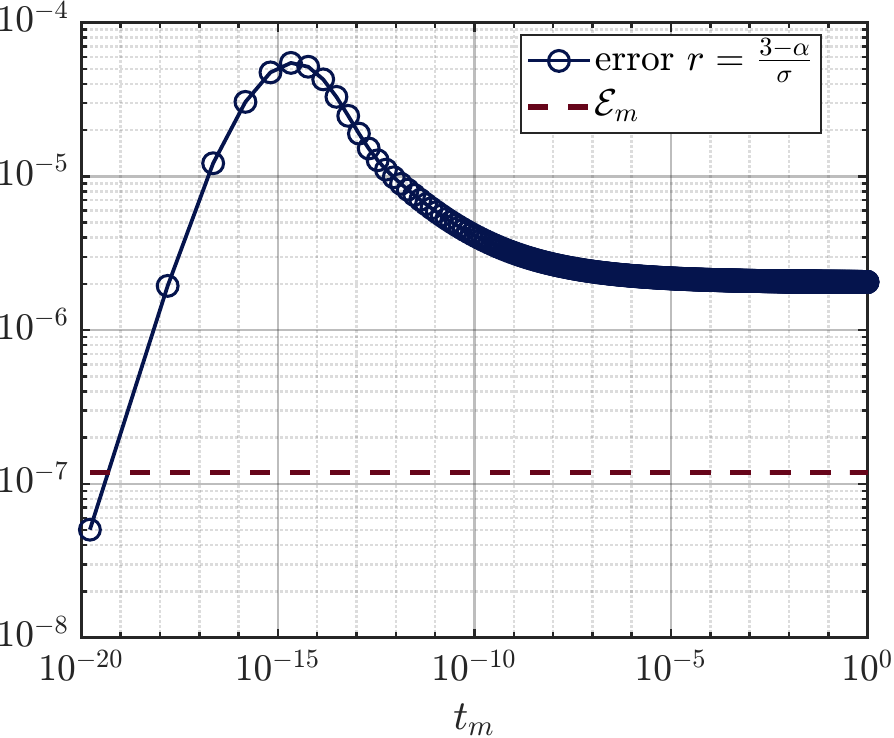}
	\end{minipage}%
	\hspace{0.8em}%
	\begin{minipage}{0.45\linewidth}
		\centering
		\includegraphics[width=\linewidth]{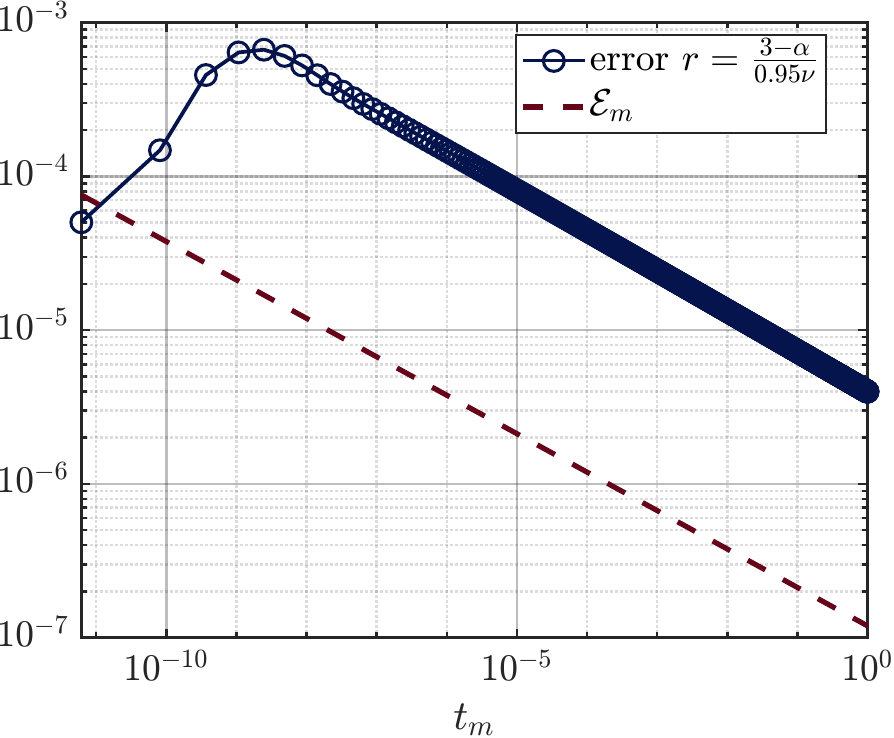}
	\end{minipage}	
	\caption{Initial-value test problem in Example \ref{ex:simple} with $\sigma = \alpha/2$: Pointwise errors for $\alpha=0.7$ and $M=1024$, cases $r=\frac{3-\alpha}{\sigma}$ (left) and $r=\frac{3-\alpha}{0.95\nu}$ (right).}
	\label{fig-simple_sigmaEalphaD2}
\end{figure}

\begin{table*}[h]
	\caption{Initial-value test problem in Example \ref{ex:simple} with $\sigma=\alpha$: Maximum nodal errors (odd rows) and computational rates $q$ in $M^{-q}$ (even rows) for $r=1,\ r=3-\alpha,\ r=\frac{3-\alpha}{\alpha}$.  \label{tb-simple-max}}
	\tabcolsep=0pt
	\begin{tabular*}{\textwidth}{@{\extracolsep{\fill}}llrrrrr@{\extracolsep{\fill}}}
		\toprule%
		& & $M=2^6$ & $M=2^7$ & $M=2^{8}$ & $M=2^{9}$ & $M=2^{10}$ \\
		\midrule
		\multirow{6}{*}{$r=1$}
		&\multirow{2}{*}{$\alpha=0.3$} & 5.2990e-02 & 4.3042e-02 & 3.4961e-02 & 2.8397e-02 & 2.3065e-02\\
		&  & 0.300      & 0.300      & 0.300      & 0.300      & \\
		&\multirow{2}{*}{$\alpha=0.5$} & 2.6825e-02 & 1.8968e-02 & 1.3413e-02 & 9.4842e-03 & 6.7063e-03\\
		&  & 0.500      & 0.500      & 0.500      & 0.500      & \\
		&\multirow{2}{*}{$\alpha=0.7$} & 1.0040e-02 & 6.1802e-03 & 3.8044e-03 & 2.3419e-03 & 1.4416e-03\\
		&  & 0.700      & 0.700      & 0.700      & 0.700      & \\
		\midrule
		\multirow{6}{*}{$r=3-\alpha$}
		&\multirow{2}{*}{$\alpha=0.3$} & 1.3007e-02 & 7.4191e-03 & 4.2317e-03 & 2.4137e-03 & 1.3767e-03\\
		&  & 0.810      & 0.810      & 0.810      & 0.810      & \\
		&\multirow{2}{*}{$\alpha=0.5$} & 2.6753e-03 & 1.1248e-03 & 4.7293e-04 & 1.9884e-04 & 8.3602e-05\\
		&  & 1.250      & 1.250      & 1.250      & 1.250      & \\
		&\multirow{2}{*}{$\alpha=0.7$} & 6.3869e-04 & 2.0923e-04 & 6.8545e-05 & 2.2455e-05 & 7.3563e-06\\
		&  & 1.610      & 1.610      & 1.610      & 1.610      & \\
		\midrule
		\multirow{6}{*}{$r=\frac{3-\alpha}{\alpha}$}
		&\multirow{2}{*}{$\alpha=0.3$} & 1.0090e-02 & 1.5527e-03 & 2.3895e-04 & 3.6773e-05 & 5.6591e-06\\
		&  & 2.700      & 2.700      & 2.700      & 2.700      & \\
		&\multirow{2}{*}{$\alpha=0.5$} & 5.6612e-04 & 1.0008e-04 & 1.7691e-05 & 3.1274e-06 & 5.5285e-07\\
		&  & 2.500      & 2.500      & 2.500      & 2.500      & \\
		&\multirow{2}{*}{$\alpha=0.7$} & 2.6829e-04 & 5.4480e-05 & 1.1063e-05 & 2.2464e-06 & 4.5617e-07\\
		&  & 2.300      & 2.300      & 2.300      & 2.300      & \\
		\botrule
	\end{tabular*}
\end{table*}

\begin{example}[Fractional-order parabolic test problem]\label{ex:parabolic}
Consider the following example with exact solution $u(x,t)=t^{\alpha}\sin(\pi x)$.
  \[
     \pt^{\alpha}_{t} u (x,t) - \partial_{x}^{2} u(x,t) =f(x,t),\ (x,t)\in(0,1)\times(0,1].
  \]
\end{example}

In Table \ref{tb-parabolic-FEM}, by applying the linear finite element method with $N=2048$, we present the maximum $L_{2}(\Omega)$ errors and the corresponding convergence rates for the fractional-order parabolic test problem in Example \ref{ex:parabolic}. The numerical results show that the convergence order in time is $3-\alpha$ when $r=\frac{3-\alpha}{\alpha}$, in agreement with the analytical results in  \eqref{co-convergence}.

\begin{table*}[h]
	\caption{Fractional-order parabolic test problem in Example \ref{ex:parabolic}: Maximum $L_{2}(\Omega)$ errors and computational rates $q$ in $M^{-q}$ (even rows) for $ r=\frac{3-\alpha}{\alpha}$ and spatial discrete by finite element method with $N=2048$. \label{tb-parabolic-FEM}}
	\tabcolsep=0pt
	\begin{tabular*}{\textwidth}{@{\extracolsep{\fill}}crrrrrr@{\extracolsep{\fill}}}
		\toprule%
		& $M=2^6$ & $M=2^7$ & $M=2^{8}$ & $M=2^{9}$ & $M=2^{10}$ \\
		\midrule
		\multirow{2}{*}{~~~$\alpha=0.3$} & 7.0850e-03 & 1.0968e-03 & 1.6894e-04 & 2.6002e-05 & 4.0016e-06\\
		& 2.692      & 2.699      & 2.700      & 2.700      & \\
		\multirow{2}{*}{~~~$\alpha=0.5$} & 3.9658e-04 & 7.0648e-05 & 1.2506e-05 & 2.2113e-06 & 3.9092e-07\\
		& 2.489      & 2.498      & 2.500      & 2.500      & \\
		\multirow{2}{*}{~~~$\alpha=0.7$} & 1.8223e-04 & 3.8208e-05 & 7.8095e-06 & 1.5879e-06 & 3.2254e-07 \\
		& 2.254      & 2.291      & 2.298      & 2.300      & \\
		\botrule
	\end{tabular*}
\end{table*}

	\begin{example}[Semilinear test problem]\label{ex:semi}
		As a test problem, consider $\mathcal{L}=-\partial_{x}^{2}$ and Allen-Cahn type nonlinearity $f(x,t)=u(x,t)-u^3(x,t)+g(x,t)$ posed in the domain $\Omega=(0,\pi)$ for $t\in[0,1]$. Here, we choose the source term $g(x,t)$ such that the exact solution is given by $u(x,t)=t^{\alpha} \sin(x^2/\pi)$.
	\end{example}
	
Finally, we numerically test our theoretical results for a semilinear subdiffusion equation. For the spatial discretization, we employ the standard linear finite element method with a sufficiently fine mesh $N=2048$. To efficiently resolve the nonlinear term $u-u^3$ at each time step, Newton's iteration method is implemented.

Table~\ref{tb-semi-FEM} presents the maximum $L_{2}(\Omega)$ errors and the corresponding temporal convergence rates for the semilinear test problem in Example~\ref{ex:semi}. The optimal grading parameter $r=\frac{3-\alpha}{\alpha}$ is used.
For $\alpha\in\{0.3,0.5,0.7\}$, the numerical results confirm that the proposed scheme achieves the expected optimal  convergence order \(3-\alpha\) in time for semilinear problems.

\begin{table*}[h]
	\caption{Semilinear test problem in Example \ref{ex:semi}: Maximum $L_{2}(\Omega)$ errors and computational rates $q$ in $M^{-q}$ (even rows) for $ r=\frac{3-\alpha}{\alpha}$ and spatial discrete by finite element method with $N=2048$. \label{tb-semi-FEM}}
	\tabcolsep=0pt
	\begin{tabular*}{\textwidth}{@{\extracolsep{\fill}}crrrrr@{\extracolsep{\fill}}}
		\toprule%
		& $M=2^5$ & $M=2^6$ & $M=2^{7}$ & $M=2^{8}$ & $M=2^{9}$ \\
		\midrule
		\multirow{2}{*}{~~~$\alpha=0.3$} & 7.1216e-02 & 1.0989e-02 & 1.6918e-03 & 2.6037e-04 & 4.0069e-05\\
		& 2.696      & 2.699      & 2.700      & 2.700      & \\
		\multirow{2}{*}{~~~$\alpha=0.5$} & 3.4769e-03 & 6.1647e-04 & 1.0904e-04 & 1.9277e-05 & 3.4077e-06 \\
		& 2.496      & 2.499      & 2.500      & 2.500      & \\
		\multirow{2}{*}{~~~$\alpha=0.7$} & 1.4210e-03 & 2.9152e-04 & 5.9330e-05 & 1.2053e-05 & 2.4478e-06 \\
		& 2.285      & 2.297      & 2.299      & 2.300      & \\
		\botrule
	\end{tabular*}
\end{table*}

\section{Conclusion}

For linear and semilinear subdiffusion equations, we have investigated
 a backward-quadratic-interpolation-based L2 scheme of order $3-\alpha$ in time.
Assuming that the exact solutions
 exhibit  an initial singularity of type $t^\sigma$,
      for any fixed $\sigma \in (0,1) \cup (1,2)$,
we have established sharp pointwise-in-time error bounds on graded temporal meshes with arbitrary degree of grading.


 \section*{Funding}
The first author
was partly funded by Research Ireland under grant number 23/RC/13506.
The second author acknowledges  support of the China Scholarship Council program (Project ID: 202406120313).

\bibliographystyle{plain}
\bibliography{ref}


\end{document}